\documentclass[10pt,twocolumn]{article}   % Enable this line and disable the 
\usepackage{color,amssymb,amsmath,amsthm}
\usepackage{graphicx}      
\usepackage{subfigure}   
\usepackage{authblk} 
\newcommand{\R}{{\mathbf{R}}}
\newcommand{\C}{{\mathbf{C}}}

\newcommand{\N}{\mathbf{N}}

\newtheorem{theorem}{Theorem}

\newtheorem{lemma}[theorem]{Lemma}
\newtheorem{proposition}[theorem]{Proposition}
\newtheorem{definition}[theorem]{Definition}
\newtheorem{remark}[theorem]{Remark}
\begin{document}


\title{Effective adiabatic control of a decoupled Hamiltonian obtained by rotating wave approximation} % Title, preferably not more 
%  \author{Nicolas Augier \footnote{Inria, Universit\'{e} C\^{o}te d'Azur, INRA, CNRS, Sorbonne Universit\'{e}, Sophia Antipolis, France}\,, 
%Ugo Boscain \footnote{Laboratoire Jacques-Louis Lions, CNRS, Inria, Sorbonne Universit\'e, Universit\'e de Paris, Paris, France}, Mario Sigalotti \footnote{Laboratoire Jacques-Louis Lions, CNRS, Inria, Sorbonne Universit\'e, Universit\'e de Paris, Paris, France}}

\author[1]{Nicolas Augier}
\author[2]{Ugo Boscain}
\author[3]{Mario Sigalotti}
\affil[1]{Inria, Universit\'{e} C\^{o}te d'Azur, INRA, CNRS, Sorbonne Universit\'{e}, Sophia Antipolis, France}
\affil[2,3]{Laboratoire Jacques-Louis Lions, CNRS, Inria, Sorbonne Universit\'e, Universit\'e de Paris, Paris, France}
%   \thanks{                           
%This work was supported 
%by the ANR projects SRGI ANR-15-CE40-0018 and
%Quaco ANR-17-CE40-0007-01.
%}                                         
%\author{Nicolas Augier\footnote{CMAP, \'Ecole Polytechnique, Institut Polytechnique de Paris}\,, 
%Ugo Boscain\;$^\S$\,, Mario Sigalotti \footnote{Sorbonne Universit\'e, Inria, Universit\'{e} de Paris, CNRS, Laboratoire Jacques-Louis Lions, Paris, France}}
%    \maketitle
%    \begin{abstract}
%In this paper
%we study up to which extent we can apply adiabatic control strategies to a quantum control model obtained by rotating wave approximation. 
%In particular, we show that, under suitable assumptions on the asymptotic regime between the %small
%parameters characterizing the rotating wave and the adiabatic approximations, the induced flow converges to the one obtained by considering the two approximations separately and by combining them formally in cascade. 
%As a consequence, we propose explicit control laws which can be used to induce desired populations transfers, robustly with respect to parameter dispersions in the controlled Hamiltonian.
%    \end{abstract}

\twocolumn[
  \begin{@twocolumnfalse}
    \maketitle
    \begin{abstract}
In this paper
we study up to which extent we can apply adiabatic control strategies to a quantum control model obtained by rotating wave approximation. 
In particular, we show that, under suitable assumptions on the asymptotic regime between the %small
parameters characterizing the rotating wave and the adiabatic approximations, the induced flow converges to the one obtained by considering the two approximations separately and by combining them formally in cascade. 
As a consequence, we propose explicit control laws which can be used to induce desired populations transfers, robustly with respect to parameter dispersions in the controlled Hamiltonian.
    \end{abstract}
  \end{@twocolumnfalse}
  ]

\section{Introduction}

A common and fruitful approach to study the controllability properties of a quantum system
%\[i \dot \psi=H(u)\psi\]
consists in 
replacing 
the controlled Hamiltonian $H(u)$ by an effective Hamiltonian  $H_{\rm eff}(\hat u)$ with more degrees of freedom, 
in such a way that the dynamics induced by $H_{\rm eff}$  
can be approximated arbitrarily well by the trajectories corresponding to $H$.

%\emph{enhancing} the controlled Hamiltonian $H(u)$, 
%characterizing the dynamics 
%which is replaced 
%by an effective Hamiltonian  $H_{\rm eff}(\hat u)$ with more degrees of freedom, 
%in such a way that the dynamics induced by $H_{\rm eff}$  
%can be followed or approximated arbitrarily well by the trajectories corresponding to $H$.  

One of the most popular  procedures to do so is based on the so-called \emph{rotating wave 
approximation} and works as follows: up to a time-dependent change of coordinates (which consists in adopting a rotating coordinate frame driven by the drift Hamiltonian $H(0)$ corresponding to the control $u=0$), 
and up to taking as control the  superposition of monochromatic pulses in resonance with the spectral gaps of the system, we can split %distinguish in 
the resulting dynamics into those terms which oscillate fast 
and those which 
evolve relatively slowly.
Then, if we assume that the control is applied on a time-interval of length $T$ (much larger than the period of the oscillating terms)
and is of  amplitude of order $1/T$, then the effect of the oscillating terms can be neglected by an averaging argument and the effect of the remaining %slowly varying 
terms is of order $1$ (see, e.g. \cite{Guerin98,Guerin99,Irish,shore-book}). % \mario{[REF: SHORE?]}.
The increase in the amount of degrees of freedom of the effective Hamiltonian obtained in this way comes from the fact that the amplitude and phase of every monochromatic pulse applied as control in the original system play the role of independent controls in the resulting effective system. We say that the original Hamiltonian is \emph{decoupled}, in the sense that it can be seen as the linear combination of Hamiltonians of smaller rank which can be controlled independently.

For a quantum control systems with several degrees of freedom, a popular control strategy 
is based on adiabatic approximation. If the controlled Hamiltonian varies slowly, 
a trajectory having as initial condition an eigenvector approximately follows the quasi-static curve 
corresponding to the eigenvectors of the non-autonomous Hamiltonian. The property  of having several degrees of freedom in the controlled Hamiltonian allows to design %non-autonomous Hamiltonian 
loops in the space of controls
whose corresponding adiabatic trajectories drive the system from an energy level  to a different one \cite{Bos,SHAP,Teu,Vitanov}.  % \mario{[REFS: Teufel, Nenciu, others, Chittaroetal]}. 
The advantage of the adiabatic control strategy (instead of, for instance, Rabi pulses) is 
that it is robust to parameter incertainties in the Hamiltonian.  In particular, it can be used to drive ensembles of quantum systems \cite{Ensemble,RouchonSarlette,shore-book,Vitanov}. %\mario{[REFS: shore, others, AugieretalSIAM]}.

It is then tempting to adopt an adiabatic control strategy to the quantum system corresponding to the decoupled Hamiltonian obtained by the rotating wave approximation. The resulting control for the original system is 
a superposition of monochromatic pulses 
whose amplitude and phase vary slowly. 
Actually, since the precision of both adiabatic and rotating wave approximations depend on 
the length of the time-interval on which the control is defined, the concatenation of the two strategies leads to a `doubly slow' control.  
More precisely, the resulting control %for the original system 
is defined on a interval of the type $[0,1/(\epsilon_1 \epsilon_2)]$, where the 
rotating wave approximation becomes more accurate as $\epsilon_1\to 0$, while the 
adiabatic strategy converges to the desired target as $\epsilon_2\to 0$. However, the convergence of the rotating wave approximation is guaranteed only when $\epsilon_2>0$ is fixed.
%, and, similarly, the adiabatic approximation is known to work only for the system obtained at the limit $\epsilon_1\to 0$. 
The concatenation of the two time-scales approximations is then not guaranteed to hold, in general, as $(\epsilon_1,\epsilon_2)\to (0,0)$. 

In \cite{CDC-nicolas} we considered this question for two-level systems with a simple structure. 
%We showed, 
%in particular, that the concatenation can fail to work for certain choices of the relative speed of convergence to 0 of $\epsilon_1$ and $\epsilon_2$. 
In particular, seting
%Namely, we considered that $\epsilon_1$ can be expressed as a function of $\epsilon_2$, say 
$\epsilon_1=\epsilon_2^\alpha$, %for some  $\alpha>0$, 
%and 
%we exhibited examples where, for some choice of $\alpha$, the cascade of the two approximations does not produce trajectories converging to the adiabatic trajectories of the decoupled Hamiltonian system. 
%We also 
we showed that the controls obtained by concatenating the two strategies do steer the system approximately close to the desired trajectories, at least when $\alpha>1$. 

Here we extend the results of  \cite{CDC-nicolas} by considering more general quantum control systems. More precisely, we consider single-input control-affine systems (i.e., we take %systems such that  
$H(u)=H_0+u H_1$, 
with $H_0$ and $H_1$ self-adjoint $n\times n$ matrices),
%$u\in \R$) in Hilbert spaces of \red{finite} dimension, %both finite and infinite, 
and we show how to identify the corresponding decoupled Hamiltonian corresponding to the rotating wave approximation. Then, we prove that, under the assumption that $\epsilon_1=\epsilon_2^\alpha$ for some   $\alpha>1$, the trajectories of the original system corresponding to the pulses obtained by the formal cascade of the two approximations converge  
to the adiabatic trajectories of the decoupled Hamiltonian. 
The proof is based on adapted quantitative averaging results for unbounded oscillating vector fields on the unit group $U(n)$, in the spirit of \cite{Kurzweil88,Liu97}.
%. We refer to \cite{Kurzweil87,Kurzweil88bis,Kurzweil88,Liu97,Sussmann1993} for more results on general  averaging results of dynamical systems with unbounded and highly oscillating inputs.
%Our result provides an estimate of the error in the special case of quantum systems.

Under some simplifying assumption on the resonances of the system (which lead to simpler expressions of the decoupled Hamiltonian), we also give explicit expressions of the pulses leading to generalized (i.e., multilevel) chirped pulses and STIRAP trajectories of the decoupled Hamiltonian.

As we recalled above, one of the main reasons to adopt an adiabatic control strategy is that it guarantees 
precious robustness properties with respect to the parameters of the controlled Hamiltonian. Our results  show that the control strategy that we propose is still robust with respect to the parameter dispersions in $H_1$, 
%(the part of the Hamiltonian directly tunable by the control $u$), 
although this is not in general true for the parameter dispersions in the drift Hamiltonian $H_0$. 
The case of parameter dispersions in $H_0$
requires a different strategy 
%and a more complicated analysis, 
and
is discussed for two-level systems in~\cite{preprint-Remi}.

%We will tackle the case of parameter  in a future work, adopting  than the direct cascade of the rotating wave and the adiabatic approximations. 

The paper is organized as follows. In Section~\ref{sec:decou} we give the general expression of the decoupled Hamiltonian obtained by rotating wave approximation for single-input bilinear %control-affine 
quantum systems.
% in Hilbert spaces of finite dimension. 
We also state the main result about the effectiveness of the concatenation of the adiabatic control %strategy 
of the decoupled Hamiltonian and the rotating wave approximation (Theorem~\ref{nlevel:thm:res}), and we 
give the general expression of the corresponding pulses. Section~\ref{sec:proof} contains the proof of Theorem~\ref{nlevel:thm:res} and the averaging results used to obtained it. 
%In Section~\ref{sec:infinite-dim} we show how the results of the previous two sections can be extended to the case of Hilbert spaces of infinite dimension. 
Finally, in Section~\ref{sec:strategies}, we apply the general %control 
construction to the case where 
the original Hamiltonian satisfy a suitable non-resonance simplifying property, and we identify some explicit adiabatic controls for the decoupled Hamiltonian that induce a complete population transfer between the energy levels of the drift Hamiltonian. 

 \section{Decoupled Hamiltonian
%quasi-resonant oscillations 
and its induced adiabatic evolution}\label{sec:decou}

%Denote by $S_n(\R)$ the space of $n\times n$ real-valued symmetric matrices. 
Fix $n\in \N$ and 
let $H_0,H_1\in i{\mathfrak u}(n)%S_n(\R)
$, where ${\mathfrak u}(n)$ denotes the Lie algebra of $n\times n$ skew-adjoint matrices. 
Consider the 
system 
\begin{equation}\label{Eq:gen}
i\frac{d\psi(t)}{dt}=(H_0+u(t) H_1)\psi(t),\qquad \psi(t)\in \C^n,%\ u(t)\in \R.
\end{equation} 
where the control $u$ takes values in $\R$.
%\mario{[$u$ in a neigh. of $0$?]}
Up to a %real-valued 
unitary change of variables, we can assume that
$H_0=\mathrm{diag}(E_j)_{j=1}^n$ with $E_1,\dots,E_n\in \R$.
%\[H_0=\mathrm{diag}(E_j)_{j=1}^n, \qquad\mbox{with }\quad 
%E_1,\dots,E_n\in \R.\]
%E_1\ge \dots \ge E_n.\]
Define 
 \[\Xi=\left\{|E_j-E_k| \mid  (j,k)\in \{1,\dots,n\}^2,\, 
(H_{1})_{j,k}
\neq 0
\right\}
\]
% \[\Xi=\left\{\sigma\in [0,+\infty) \mid  \exists (j,k)\in \{1,\dots,n\}^2\mbox{ s.t. } \left. \begin{aligned} \sigma=|E_j-E_k|\\ 
%%\langle \phi_{k_1}, H_{1} \phi_{k_2}\rangle
%(H_{1})_{j,k}
%\neq 0
%\end{aligned}\right\}
%\right.
%\]
i.e., $\Xi$ is the set of nonnegative spectral gaps of $H_0$ corresponding to a direct coupling by the controlled Hamiltonian $H_1$. 
Let us now use the spectral gaps  in $\Xi$ to identify 
a controlled Hamiltonian which corresponds to a decoupling
of $H_1$. % by means of \mario{mock} controls. 
The effectivity of 
such a decoupling 
is illustrated in the next sections.

For $\sigma\in \Xi$, let 
$\mathcal{R}_{\sigma}\subset  \{1,\dots,n\}^2$ and 
$H_1^\sigma\in i\mathfrak{u}(n)$ be defined by
\begin{align}
\mathcal{R}_{\sigma}&=\left\{(j,k)%\in \{1,\dots,n\}^2 
\mid  %\left.\begin{aligned}
E_{j}-E_{k}=\sigma,\,
% \\
 (H_{1})_{j,k}\neq 0
%\end{aligned} \right.
\right\},\label{Rsigma}\\
(H_1^\sigma)_{j,k}&=
\begin{cases}
(H_1)_{j,k}&\mbox{if }|E_j-E_k|=\sigma,\\
0&\mbox{otherwise}.
\end{cases}
\nonumber
\end{align}
The idea of the decoupling is to consider that  each of the matrices $H_1^\sigma$ can be controlled autonomously. More precisely, define 
the \emph{decoupled Hamiltonian}
$H_{\rm d}:\R^{n}\times \R^{\Xi}\to i{\mathfrak u}(n)%S_n(\R)
$ 
 by
\begin{equation}\label{H-decoupled}
H_{\rm d}(\delta,w)=
%\sum_{j=1}^n \delta_j e_{jj}+\sum_{(j,k)\in \Xi} w_{jk}\mario{(H_{1})_{k,j}(e_{jk}+e_{kj})}%J_{kj}
\sum_{j=1}^n \delta_j e_{jj}+\sum_{\sigma\in \Xi} w_{\sigma}H_{1}^\sigma,
\end{equation}
where $\R^{\Xi}$ denotes the set of real vectors $(w_{\sigma})_{\sigma\in \Xi}$, and, for every $j,k\in \{1,\dots,n\}$, $e_{jk}$ is the $n\times n$ matrix whose $(j,k)$-coefficient is equal to $1$ and the others are equal to $0$.

\subsection{Adiabatic evolution of the decoupled Hamiltonian}

%\subsection{Decoupling %of the controls 
%by rotating wave approximation %
%}

%\mario{[THIS SUBSECTION MAY BE DROPPED (but some definitions should be kept)]}

For every $k\in\{1,\dots,n\}$ consider a  smooth function $\varphi_k:[0,1]\to \R$ such that $\varphi_k(0)=0$.
Consider, in addition, a family of smooth functions 
$v_{\sigma}:[0,1]\to \R$,  $\sigma\in \Xi$, and, if $0\not\in \Xi$, set $v_0\equiv 0$.

Assume, moreover, that 
\[\tag{$4$}\begin{aligned}
&\varphi_j-\varphi_k=\varphi_p-\varphi_q\\
&\mbox{if }\exists\, \sigma\in \Xi
 \mbox{ s.t. }(j,k),(p,q)\in \mathcal{R}_{\sigma}\mbox{ and }v_\sigma\ne 0, \label{eq:nonres-phi}
\end{aligned}
\]
where by $v_\sigma\ne 0$ we mean that $v_\sigma$ is not identically equal to zero. 
Concerning the solvability of the constraint \eqref{eq:nonres-phi}, notice that, for every $\phi\in C^\infty([0,1],\R)$ such that $\phi(0)=0$, the choice $\varphi_k=\phi$ for every $k=1,\dots,n$ satisfies \eqref{eq:nonres-phi}.
In general, there exists a linear subspace $L$ of $\R^n$ of dimension at least one 
such that
\eqref{eq:nonres-phi} is satisfied if and only if 
the vector valued function $\varphi=(\varphi_j)_%{j\in\pi\Xi}
{j=1}^n:[0,1]\to \R^n$ takes values in $L$.

Given $\varphi=(\varphi_j)_%{j\in\pi\Xi}
{j=1}^n:[0,1]\to L$ and $\sigma\in \Xi$
such that $v_\sigma\ne 0$, we 
define 
\[
\hat \varphi_\sigma=\varphi_j-\varphi_k,\qquad\mbox{for }(j,k)\in \mathcal{R}_{\sigma}. 
\]

Let $\alpha>1$ and define 
\begin{equation}\label{eq:ueps}
\begin{aligned} &u_{\epsilon}(t)=
\epsilon v_0(\epsilon^{\alpha+1} t)\\
&+2\epsilon \sum_{\sigma\in \Xi\setminus\{0\}} v_{\sigma}(\epsilon^{\alpha+1} t)
\cos\left(\sigma t+\frac{\hat\varphi_\sigma(\epsilon^{\alpha+1}t)}{\epsilon}\right).\end{aligned}
\end{equation} 
%\mario{[CHECK THE ORDER IN $\eps$ for the term $\frac{\hat\varphi_\sigma(\epsilon^{\alpha+1}t)}{\epsilon}$]}

%Define for $t\in [0,\frac1\epsilon]$,
%\[\bar u_{\epsilon}(t)=\epsilon\left(v_0(\epsilon t)+2\sum_{\sigma\in \Xi\setminus\{0\}} v_{\sigma}(\epsilon t)
%\cos\left(\frac{\sigma t}{\epsilon}+\hat\varphi_\sigma(\epsilon t)
%\right)\right),\]
%and denote %$v_{jk}=v_{[j,k]}$,  
%$v=(v_{\sigma})_{\sigma\in \Xi}$.

Denote 
$v=(v_{\sigma})_{\sigma\in \Xi}$ and 
let $h_{\rm d}:[0,1]\to  i{\mathfrak u}(n)%S_n(\R)
$ be the non-autonomous Hamiltonian 
%For every $\tau\in [0,1]$, define $h_{\rm d}(\tau)\in \mathfrak{u}(n)$ such that 
\begin{equation}\label{hatH}
h_{\rm d}(\tau)=H_{\rm d}(-\varphi'(\tau),v(\tau)),\qquad \tau\in[0,1].
\end{equation}
Before stating our main result on the approximation by solutions of system \eqref{Eq:gen} of the adiabatic trajectories induced by the non-autonomous Hamiltonian $h_{\rm d}$, let us introduce the %definition of 
following gap condition. % for non-autonomous Hamiltonians.
%Such condition plays a crucial role in proving adiabatic approximation results. 
\begin{definition}
For $h\in C^\infty([0,1],i{\mathfrak u}(n)%S_n(\R)
)$, 
let
%define $\Lambda(\tau)=\mathrm{diag}(\lambda_j(\tau))_{j\in \{1,\dots,n\}}$ where  
$\lambda_1,\dots,\lambda_n:[0,1]\to \R$ %is 
be smooth and such that $\{\lambda_1(\tau),\dots,\lambda_n(\tau)\}$ is the spectrum of $h(\tau)$ for every $\tau\in [0,1]$.
We say that $h$ satisfies a \emph{gap condition} if there exists $C>0$ such that 
\begin{equation*}\label{gap}
\tag{GAP}
\begin{aligned}
\forall j,\ell \in \{1,\dots,n\}\mbox{ s.t. }j\ne \ell,\; \\
\forall \tau\in [0,1],\qquad |\lambda_j(\tau) -\lambda_{\ell}(\tau)|\ge C.
\end{aligned}
\end{equation*}
Given a nonnegative integer $\kappa$, 
we say that $h$ satisfies a \emph{$\kappa$-th order gap condition} if  
\begin{equation*}\label{kgap}
\tag{$\kappa$-GAP}
\begin{aligned}
\forall j,\ell \in \{1,\dots,n\}\mbox{ s.t. }j\ne \ell,  \forall \tau\in [0,1],\; \\
\begin{cases}\exists r\in\{0,\dots,\kappa\}&\mbox{s.t. }\frac{d^r (\lambda_{\ell}-\lambda_j)}{d\tau^r}(\tau)\neq 0,\\
\lambda_p(\tau)\neq\lambda_{j}(\tau)&\mbox{for }p\ne j,\ell.
\end{cases}
\end{aligned}
\end{equation*}
\end{definition} 
\begin{remark}
Notice that a $0$-th order gap condition is nothing else that a gap condition. 
Concerning the regularity of eigenvalues, it is known if $h\mapsto h(\tau)$ is $C^\infty$, and if the order of contact of any two unequal eigenvalues of $h$ is finite, then all the eigenvalues and all the eigenvectors 
$h(\tau)$ 
can be chosen to be $C^\infty$ with respect to $\tau$ 
\cite[Theorem 7.6]{AKML1998}. 
\end{remark}

%Let $\alpha>1$ and define 
%\begin{equation}\label{eq:ueps-doppio}
%\begin{aligned} &u_{\epsilon}(t)=
%\epsilon v_0(\epsilon^{\alpha+1} t)\\
%&+2\epsilon \sum_{\sigma\in \Xi\setminus\{0\}} v_{\sigma}(\epsilon^{\alpha+1} t)
%\cos\left(\frac{\sigma t}{\epsilon^{\alpha+1}}+\frac{\hat\varphi_\sigma(\epsilon^{\alpha+1}t)}{\epsilon}\right)\end{aligned} ,
%\end{equation} 
%where 
%$v$, $\varphi$, and $\hat \varphi$ are as in the previous section. 
%Letting $h_{\rm d}$ be defined as in \eqref{hatH}, we have the following. 

Our main result is the following. 

\begin{theorem}\label{nlevel:thm:res}
Set $\psi_0\in \C^n$ and let $\psi_{\epsilon}(\cdot)$ be the solution of
\begin{equation*}%\label{}
i\frac{d\psi_{\epsilon}(t)}{dt}=\left(H_0+u_{\epsilon}(t)H_1\right) \psi_{\epsilon}(t),\qquad 0\le t\le\frac1{\epsilon^{\alpha+1}},
\end{equation*} 
with $\psi_{\epsilon}(0)=\psi_0$.
% where by a slight abuse of notation, we write \[u_{\epsilon}(\tau)=2\sum_{(k,q)\in \Xi} v_{kq}(\tau)\cos(\frac{(E_k-E_q)\tau}{\epsilon}+\left(\varphi_k(\tau)-\varphi_q(\tau)\right)).\]
%
Define $\hat{\Psi}_{\epsilon}$ %:[0,1/\epsilon]\to \C^n$ 
as the solution of
\begin{equation}
\frac{d\hat{\Psi}_{\epsilon}(s)}{ds}=h_{\rm d}(\epsilon s) \hat{\Psi}_{\epsilon}(s),\qquad 0\le s\le \frac1\epsilon,
\end{equation} 
with $\hat{\Psi}_{\epsilon}(0)=\psi_0$.
Assume  that $h_{\rm d}$ satisfies a $\kappa$-th order gap condition for some nonnegative integer $\kappa$.
Then
 \[\left\|\psi_{\epsilon}\left(\frac{\tau}{\epsilon^{\alpha+1}}\right) -V_{\epsilon}%\left(\frac{s}{\epsilon}\right)
 (\tau)\hat{\Psi}_{\epsilon}\left(\frac{\tau}{\epsilon}\right)\right\|<c \epsilon^{\min(\frac1{\kappa+1},\alpha-1)}\] 
for $0\le \tau\le 1$, where
  $V_{\epsilon}(\tau)=\mathrm{diag}\left(e^{-i(\frac{E_{j}\tau}{\epsilon^{\alpha+1}}+\frac{\varphi_j(\tau)}{\epsilon})}\right)_{j=1}^n$
   and 
 $c>0$ is independent of $\tau\in [0,1]$ and $\epsilon>0$.
  \end{theorem}

The proof is postponed to the next section, where the required preliminary technical results are 
obtained.

\section{Approximation results and proof of Theorem~\ref{nlevel:thm:res}}\label{sec:proof}

In this section, we prove several averaging results leading to the proof of Theorem~\ref{nlevel:thm:res}.
Up to a suitable time-rescaling and change of coordinates, the proof is based on 
an adiabatic approximation where  high-order oscillating terms
are proved to be negligible.  

\subsection{Negligible high-order terms in adiabatic approximation}
Consider $A\in C^\infty([0,1],\mathfrak{u}(n))$.
%$A(\tau)\in \mathfrak{u}(n)$ and a perturbation $B_{\epsilon}(\tau)\in \mathfrak{u}(n)$ whose dependence in $\tau\in [0,1]$ is $C^{\infty}$. 
%Up to a time-rescaling from the interval $[0,1/\epsilon]$ to the interval $[0,1]$, 
The 
adiabatic evolution corresponding to $A$ is described by the equation
\begin{equation}\label{eq;adi}
\frac{d\hat{X}_{\epsilon}(s)}{ds}= A(\epsilon s)\hat{X}_{\epsilon}(s),\quad 0\le s\le \frac{1}{\epsilon}.
\end{equation}
Our goal is to understand under which conditions on a perturbation term  $(B_{\epsilon}(\cdot))_{\epsilon>0}$
%\footnote{{\color{magenta}the term $1/\epsilon$ should be factorized out of $B_\epsilon$. I do not if it can be better to  write everything in the time-scale $[0,1/\epsilon]$}}
 the flows of 
\eqref{eq;adi} and  
\begin{equation}\label{eq;adi-eps}
\frac{dX_{\epsilon}(s)}{ds}= \left(A(\epsilon s)+B_{\epsilon}(\epsilon s)\right)X_{\epsilon}(s),\quad 0\le s\le \frac{1}{\epsilon},
\end{equation}
 are arbitrarily close, as $\epsilon\to 0$.
 %, uniformly with respect to $s\in [0,1/\epsilon]$.

\begin{definition}\label{B}
Given $\alpha>1$, denote by $S(\alpha)$  the set of families $(B_{\epsilon})_{\epsilon>0}$ of functions in $C^{\infty}([0,1],\mathfrak{u}(n))$ such that %{\color{magenta}[why not just once with $j\le k$? did I copy something wrong? in the cdc the diagonal term was zero]}
\begin{itemize}
\item for every $j\in \{1,\dots,n\}$,  there exist $\beta_{jj}\in\R\setminus\{0\}$ and 
$v_{jj},h_{jj}\in C^{\infty}([0,1],\R)$ such that $(B_{\epsilon}(\tau))_{jj}=%-
%\frac{i}{\epsilon}
i 
v_{jj}(\tau)\cos(\frac{\beta_{jj}\tau}{\epsilon^{\alpha+1}}+\frac{h_{jj}(\tau)}{\epsilon})$ and every $\tau\in [0,1]$,
\item 
for every $1\le j<k\le n$ there exist $\beta_{jk}\in\R\setminus\{0\}$ and 
$v_{jk},h_{jk}\in C^{\infty}([0,1],\R)$ such that $(B_{\epsilon}(\tau))_{jk}=%-
%\frac{i}{\epsilon}
i
v_{jk}(\tau)e^{i(\frac{\beta_{jk}\tau}{\epsilon^{\alpha+1}}+\frac{h_{jk}(\tau)}{\epsilon})}$ for every $\tau\in [0,1]$.
\end{itemize} 
\end{definition}

\begin{theorem}\label{Drive}
Consider $A\in C^\infty([0,1],\mathfrak{u}(n))$ and a finite sum $(B_{\epsilon})_{\epsilon>0}$  of elements belonging to $S(\alpha)$ with $\alpha>1$. Assume that $A(\cdot)$ satisfies 
a $\kappa$-th order gap condition for some nonnegative integer $\kappa$.
%(\ref{gap}). 
Fix $X_0\in \C^n$. % independent of $\epsilon$.
Let  $\hat{X}_{\epsilon}$ and  $X_{\epsilon}$ be  the solutions of, respectively, 
 \eqref{eq;adi} and \eqref{eq;adi-eps} with 
%$\frac{dX_{\epsilon}(\tau)}{d\tau}= \left(\frac{1}{\epsilon}A(\tau)+B_{\epsilon}(\tau)\right)X_{\epsilon}(\tau)$
% such that 
$\hat{X}_{\epsilon}(0)=X_{0}$ and 
%be the solution of 
% $\frac{d\hat{X}_{\epsilon}(\tau)}{d\tau}= \frac{1}{\epsilon}A(\tau)\hat{X}_{\epsilon}(\tau)$ 
%such that 
$X_{\epsilon}(0)=X_{0}$.
Then there exists $c>0$ %independent of $\tau,\epsilon$ 
such that  
$ \|\hat{X}_{\epsilon}(s)-X_{\epsilon}(s)\|\leq c \epsilon^{\min(\frac{1}{\kappa+1},\alpha-1)}$
for every $s \in [0,1/\epsilon]$ and $\epsilon>0$.
\end{theorem}
 Before proving Theorem~\ref{Drive}, let us show how it can be used to deduce Theorem~\ref{nlevel:thm:res}.
 
 % \begin{pf}[
 {\bf Proof of Theorem~\ref{nlevel:thm:res}.} %]
Let us introduce the notation $\hat \Xi$ for $\Xi\cup(-\Xi)$, and, for $\sigma\in \hat \Xi$, let $\hat H_1^\sigma$ be the matrix such that $(\hat H_1^\sigma)_{j,k}=(H_1)_{j,k}$ if $E_j-E_k=\sigma$ and $0$ otherwise.  
Moreover, for $\sigma\in \hat \Xi$ such that $v_{|\sigma|}\ne 0$
and $j,k$ such that $E_j-E_k=\sigma$ and $(H_1)_{j,k}\ne 0$, set
\[\chi_{\sigma}(\tau)=\frac{\sigma\tau}{\epsilon^{\alpha+1}}+\frac{\varphi_j(\tau)-\varphi_k(\tau)}{\epsilon},\qquad \sigma\in \hat \Xi,\]
and notice that $\chi_{-\sigma}(\tau)=-\chi_{\sigma}(\tau)$.

Define $\Psi_{\epsilon}(\tau)=V^*_{\epsilon}(\tau)\psi_{\epsilon}(\tau/\epsilon^{\alpha+1})$, $\tau\in [0,1]$, and notice that $\Psi_{\epsilon}(0)=\psi_0$, since $\varphi(0)=0$.
Since diagonal matrices commute, we easily get that $\Psi_{\epsilon}$ satisfies 
\begin{equation}
i\frac{d\Psi_{\epsilon}(\tau)}{d\tau}=-\frac1\epsilon\mathrm{diag}(\varphi'(\tau))\Psi_{\epsilon}(\tau)+
C_{\epsilon}(\tau)
\Psi_{\epsilon}(\tau),
\label{toberearranged}
\end{equation}
where for $\tau\in [0,1]$,
\begin{align*}C_{\epsilon}(\tau)=&\frac{u_\epsilon(\tau/\epsilon^{\alpha+1})}{\epsilon^{\alpha+1}}
V^*_{\epsilon}(\tau)H_1 V_{\epsilon}(\tau)\\
=&\frac{u_\epsilon(\tau/\epsilon^{\alpha+1})}{\epsilon^{\alpha+1}}\sum_{\sigma\in \hat{\Xi}} 
e^{i\chi_\sigma(\tau)} \hat H_1^\sigma.
\end{align*}

%\begin{align*}
%C_{\epsilon}(\tau)&=\frac{u_\epsilon(\tau/\epsilon^{\alpha+1})}{\epsilon^{\alpha+1}}
%V^*_{\epsilon}(\tau)H_1 V_{\epsilon}(\tau)\\&=\frac{u_\epsilon(\tau/\epsilon^{\alpha+1})}{\epsilon^{\alpha+1}}\sum_{\sigma\in \hat{\Xi}} 
%e^{i\chi_\sigma(\tau)} \hat H_1^\sigma.
%\end{align*}
Notice now that
\begin{align*}
\frac{u_\epsilon(\tau/\epsilon^{\alpha+1})}{\epsilon^{\alpha+1}}&
%=\frac2\epsilon\sum_{[p,q]\in \mathcal{R}}%\frac{v_{pq}(\tau)}{r_{pq}}
%v_{pq}(\tau)\cos(\chi_{pq}(\tau))%\\&
%=\frac1\epsilon\sum_{[p,q]\in \mathcal{R}}%\frac{v_{pq}(\tau)}{r_{pq}}
%v_{pq}(\tau)(e^{i \chi_{pq}(\tau)}+e^{-i \chi_{pq}(\tau)}).
=\frac{v_0(\tau)}{\epsilon}+\frac1\epsilon\sum_{\sigma\in \Xi\setminus\{0\}}%\frac{v_{pq}(\tau)}{r_{pq}}
 v_{\sigma}(\tau)
 (e^{i \chi_{\sigma}(\tau)}+e^{-i \chi_{\sigma}(\tau)}).
\end{align*}
In order to rewrite system \eqref{toberearranged} in the form 
$ \frac{d\Psi_{\epsilon}(\tau)}{d\tau}=\frac{1}{\epsilon}\left(A(\tau)+B_{\epsilon}(\tau)\right) \Psi_{\epsilon}(\tau)$,
define
\[\begin{aligned}A(\tau)={}&
  i \mathrm{diag}(\varphi'(\tau)) \\
 &-i\left(v_0(\tau)\hat H_1^0+\sum_{\sigma\in \Xi\setminus\{0\}}v_\sigma(\tau)(\hat H_1^\sigma+\hat H_1^{-\sigma})\right)\\
 ={}&i \mathrm{diag}(\varphi'(\tau))+\sum_{\sigma\in \Xi}v_\sigma(\tau) H_1^\sigma
% \\
=
%{}&
-i h_{\rm d}(\tau),
 \end{aligned}
\] 
and 
 \[\begin{aligned} B_\epsilon(\tau)=&-i v_0(\tau)\sum_{\sigma\in \hat \Xi\setminus\{0\}}e^{i\chi_\sigma(\tau)} \hat H_1^\sigma\\
 &-i\sum_{\sigma\in \Xi\setminus\{0\}}v_{\sigma}(\tau)(e^{2i\chi_\sigma(\tau)} \hat H_1^\sigma+e^{-2i\chi_{\sigma}(\tau)}\hat H_1^{-\sigma})\\
 & -i\sum_{\sigma\in \Xi\setminus\{0\}}v_{\sigma}(\tau)\sum_{\hat \sigma\in \Xi\setminus\{\pm \sigma\}}\Big(e^{i(\chi_{\hat\sigma}(\tau)+\chi_\sigma(\tau))} \hat H_1^{\hat \sigma}\\
 &+e^{i(-\chi_{\hat\sigma}(\tau)+\chi_\sigma(\tau))}\hat H_1^{-\hat\sigma}\Big).
\end{aligned}\]
One easily checks 
that $(B_{\epsilon})_{\epsilon>0}$ is a finite sum of elements of $S(\alpha)$.
 By applying Theorem~\ref{Drive}, we get that   
 $\|\Psi_{\epsilon}(\tau) -\hat{\Psi}_{\epsilon}(\tau)\|<c \epsilon^{\min(\frac{1}{\kappa+1},\alpha-1)}$ for every $\tau\in [0,1]$
 for some %constant 
$c>0$ independent of $\tau\in [0,1]$ and $\epsilon>0$.
% where $c>0$ is independent of $(\tau,\epsilon)$.
 % 
% \end{pf}

 \subsection{Proof of Theorem~\ref{Drive}}\label{sec:Drive}

The results in this section have been presented, in a preliminary version, in \cite{CDC-nicolas}.
 The proof of Theorem~\ref{Drive} is split in three steps. 
In the first of such steps, 
we consider an oscillating (possibly unbounded) perturbation term 
on a bounded time-interval and we give 
a condition  ensuring its negligibility 
in terms of the asymptotic behavior of its iterated integrals. 
%that it is negligible while integrating a smooth Hamiltonian . 

\begin{proposition}\label{avnb}
Let $D$ and $(M_{\epsilon})_{\epsilon>0}$ be in $C^\infty([0,1],\mathfrak{u}(n))$.
Assume that
$\int_0^\tau M_{\epsilon}(\vartheta)d\vartheta=O(\epsilon)$
and that there exists  $\eta>0$ such that
\[\int_0^\tau \left|M_{\epsilon}(\vartheta)\right|\left|\int_0^\vartheta M_{\epsilon}(\theta)d\theta\right| 
d\vartheta=O(\epsilon^\eta),\] 
both estimates being uniform with respect to $\tau\in [0,1]$. 
%Set $D_{\epsilon}=D+M_{\epsilon}$.
Denote the flow of the equation $\frac{dx(\tau)}{d\tau}=D(\tau)x(\tau)$ %at time $\tau$  
by $P_{\tau}\in {\rm U}(n)$ and the flow of the equation $\frac{dx(\tau)}{d\tau}=(D(\tau)+M_{\epsilon}(\tau))%D_{\epsilon}(\tau)
x(\tau)$ 
%at time $\tau$  
by $P_{\tau}^{\epsilon}\in {\rm U}(n)$.
%Then $$\overrightarrow{exp}(\int_0^\tau A_{\epsilon}(s)ds)=\overrightarrow{exp}(\int_0^\tau A(s) ds)+O(\epsilon^{\min(k,1)}).$$
Then %we have 
$P_{\tau}^{\epsilon}=P_{\tau}+O(\epsilon^{\min(\eta,1)})$,
uniformly with respect to $\tau\in  [0,1]$.
\end{proposition}
\begin{proof}
Let 
%Under the hypotheses of the theorem, there exists 
$K>0$ be such that 
 $|\int_0^\tau M_\epsilon(\vartheta) d\vartheta|<K \epsilon$ for every $\tau\in [0,1]$ and denote by $Q_{\tau}^{\epsilon}$ the flow associated with $M_{\epsilon}$.
Hence, $Q_{\tau}^{\epsilon}=\mathrm{Id}+\int_0^\tau M_{\epsilon}(\vartheta)Q_{\vartheta}^{\epsilon} d\vartheta$.
By integration by parts, 
%\begin{align*}
$Q_{\tau}^{\epsilon}=%&
\mathrm{Id}+\left(\int_0^\tau M_{\epsilon}(\vartheta) d\vartheta\right) Q_{\tau}^{\epsilon} 
%\\&
- \int_0^\tau \left(\int_0^\vartheta M_{\epsilon}(\theta) d\theta\right)M_{\epsilon}(\vartheta)Q_{\vartheta}^{\epsilon} d\vartheta$.
%\end{align*}
%For every $\tau$ in $[0,1]$ we have $A_{\epsilon}(\tau)\in \mathfrak{u}(n)$. Hence 
Moreover,
$Q_{\tau}^{\epsilon}$ is bounded uniformly with respect to $(\tau,\epsilon)$, since it evolves in ${\rm U}(n)$.
Hence,
%By the triangular inequality, we get 
%\begin{align*} 
$\left|Q_{\tau}^{\epsilon}-\mathrm{Id}\right|%&
\leq  
% \left|\int_0^\tau M_{\epsilon}(s) ds\right| \left|Q_{\tau}^{\epsilon}\right| + \int_0^\tau  \left|\int_0^s M_{\epsilon}(\theta) d\theta\right|\left| M_{\epsilon}(s)Q_{s}^{\epsilon}\right| ds\\
%\leq& 
C_1 \epsilon + C_2 \epsilon^\eta$,
%\end{align*}
where $C_1,C_2$ are positive constants which do not depend on $(\tau,\epsilon)$. %Hence, 
We deduce that
$Q_{\tau}^{\epsilon}=\mathrm{Id}+O(\epsilon^{\min(\eta,1)})$, uniformly with respect to $\tau \in [0,1]$. %, where $q=\min(\eta,1)$.
By the variations formula (see, e.g., \cite[Section~2.7]{Agra}), $P_{\tau}^{\epsilon}=Q_{\tau}^{\epsilon}W_{\tau}^{\epsilon}$, where $W_{\tau}^{\epsilon}\in {\rm U}(n)$ is the flow of the equation $\frac{dx(\tau)}{d\tau}=(Q_{\tau}^{\epsilon})^{-1} D(\tau) Q_{\tau}^{\epsilon}x(\tau)$. % at time $\tau$. 
By the previous estimate, we have $(Q_{\tau}^{\epsilon})^{-1} D(\tau) Q_{\tau}^{\epsilon}=D(\tau)+O(\epsilon^{\min(\eta,1)})$ uniformly with respect to $\tau \in [0,1]$.
By an easy application of Gronwall's Lemma, we get that $W_{\tau}^{\epsilon}=P_{\tau}+O(\epsilon^{\min(\eta,1)})$  and we can conclude.
\end{proof}

The second step of the proof of Theorem~\ref{Drive} consists in the following lemma, which 
will be used to apply Proposition~\ref{avnb} to a suitable reformulation of Equation~\eqref{eq;adi-eps}.

\begin{lemma}\label{conjug}
Let $\alpha>1$
 and $(B_\epsilon)_{\epsilon>0}$ be a finite sum of elements in $S(\alpha)$.
 Fix $P\in C^\infty([0,1],{\rm U}(n))$ and $\Gamma=\mathrm{diag}(\Gamma_j)_{j=1}^n$ with 
$\Gamma_j\in C^\infty([0,1],\R)$, $j=1,\dots,n$.
For every $\epsilon>0$ and $\tau\in [0,1]$, 
define %the matrix 
\begin{equation}
M(P,\Gamma,\epsilon)(\tau)=e^{i\frac{\Gamma(\tau)}{\epsilon}}P^{*}(\tau)B_{\epsilon}(\tau)P(\tau)e^{-i\frac{\Gamma(\tau)}{\epsilon}}.\label{def:MP}
\end{equation}
Then 
\begin{equation}\label{eq:stbp}
\int_0^\tau M(P,\Gamma,\epsilon)(\vartheta) d\vartheta=O(\epsilon^{\alpha})
\end{equation}
 and 
\[\int_0^\tau \left|M(P,\Gamma,\epsilon)(\vartheta)\right|\left|\int_0^\vartheta M(P,\Gamma,\epsilon)(\theta)d\theta\right| 
d\vartheta=O(\epsilon^{\alpha-1}),\] 
both estimates being uniform with respect to $\tau\in [0,1]$.
\end{lemma}

\begin{proof}
First 
notice that, since $B_\epsilon(\tau)=O(1/\epsilon)$, then $M(P,\Gamma,\epsilon)(\tau)=O(1/\epsilon)$, both estimates being uniform with respect to $\tau\in [0,1]$. Hence, it is enough to prove~\eqref{eq:stbp}.
%
%For $j,k\in \{1,\dots,n\}$ denote by $e_{jk}$  the $n\times n$ matrix 
%such that $(e_{jk})_{\ell,m}$ is equal to $1$ if $(\ell,m)=(j,k)$ and $0$ otherwise.
%whose coefficient $(j,k)$ is equal to $1$ and all others are equal to $0$.
By linearity, it is enough to prove that, for every $j,k\in \{1,\dots,n\}$, every $\beta\in \R\setminus\{0\}$, and every
$v,h\in C^\infty([0,1],\R)$, 
 the matrix 
$C_{\epsilon}(\tau)=\frac{1}{\epsilon}v(\tau)e^{i%\left
(\frac{\beta\tau}{\epsilon^{\alpha+1}}+\frac{h(\tau)}{\epsilon}%\right
)}e_{jk}$  
%where $e_{jk}$ is the matrix whose coefficient $(j,k)$ is equal to $1$ and others are equal to $0$, 
satisfies
\[\int_0^\tau 
e^{i\frac{\Gamma(\vartheta)}{\epsilon}}P^{*}(\vartheta)C_{\epsilon}(\vartheta)P(\vartheta)e^{-i\frac{\Gamma(\vartheta)}{\epsilon}}
%v_{j\ell}(s)  p_{\ell k}(s)\bar{p}_{jq}(s)e^{\frac{i}{\epsilon} (\Gamma_q(s)-\Gamma_k(s))}e^{i(\frac{\beta_{j\ell}s}{\epsilon^{\alpha+1}}+\frac{h_{j\ell}(s)}{\epsilon})}
d\vartheta=O(\epsilon^{\alpha+1}).\]
Denoting $p_{\ell m}(\tau)=(P(\tau))_{\ell,m}$ for every $\ell,m\in \{1,\dots,n\}$, we have that
\begin{align*}
&e^{i\frac{\Gamma(\tau)}{\epsilon}}P^{*}(\tau)C_{\epsilon}(\tau)P(\tau)e^{-i\frac{\Gamma(\tau)}{\epsilon}}=\\
&\frac{v(\tau)}{\epsilon}e^{i(\frac{\beta\tau}{\epsilon^{\alpha+1}}+\frac{h(\tau)}{\epsilon})}
\sum_{\ell,m=1}^n \overline{p_{j\ell}(\tau)} p_{km}(\tau)e^{\frac{i}{\epsilon} (\Gamma_\ell(\tau)-\Gamma_m(\tau))}e_{\ell m}.
\end{align*}

We conclude the proof by showing that, 
for every  $a\in C^{\infty}([0,1],\R)$, we have  
\begin{equation}\label{eq:ordera+1}
\int_0^\tau a(\vartheta) e^{i(\frac{\beta \vartheta}{\epsilon^{\alpha+1}}+\frac{h(\vartheta)}{\epsilon})} d\vartheta=O(\epsilon^{\alpha+1}),
\end{equation}
 uniformly with respect to $\tau\in [0,1]$.
 Integrating by parts, for every $\tau \in [0,1]$,
\begin{align*}
&\int_0^\tau a(\vartheta) e^{i(\frac{\beta \vartheta}{\epsilon^{\alpha+1}}+\frac{h(\vartheta)}{\epsilon})} d\vartheta \\
&=i   \frac{\epsilon^{\alpha+1}}{\beta}\int_0^\tau e^{i \frac{\beta  \vartheta}{\epsilon^{\alpha+1}}} \left( a'(\vartheta)+i \frac{h'(\vartheta)}{\epsilon}a(\vartheta)\right)e^{i\frac{h(\vartheta)}{\epsilon}}d\vartheta\\
&+ \left[-i \frac{ \epsilon^{\alpha+1}}{\beta} e^{i\frac{\beta \vartheta}{\epsilon^{\alpha+1}}}a(\vartheta)e^{i\frac{h(\vartheta)}{\epsilon}}\right]_{0}^{\tau}\\
&=-\frac{\epsilon^{\alpha}}{\beta} \int_0^\tau h'(\vartheta)a(\vartheta) e^{i (\frac{\beta \vartheta}{\epsilon^{\alpha+1}}+\frac{h(\vartheta)}{\epsilon})}d\vartheta +O(\epsilon^{\alpha+1}).
\end{align*}
Iterating the integration by parts on the integral term $\lceil\frac{1}{\alpha}\rceil$ more times, we 
obtain \eqref{eq:ordera+1}.
\end{proof}

Let us now 
%We are now ready to 
prove Theorem~\ref{Drive}. This is done by providing an explicit expression of the adiabatic evolution %corresponding to the leading term $A$ of the flow 
of \eqref{eq;adi}.
We actually prove in the next proposition that, under the assumptions of Theorem~\ref{Drive}, the leading term of the flow of \eqref{eq;adi-eps} does not depend on $(B_\epsilon)_{\epsilon>0}$. Since equation \eqref{eq;adi} corresponds to the case $B_\epsilon\equiv 0$ for every $\epsilon>0$, one deduces Theorem~\ref{Drive} simply by triangular inequality.

\begin{proposition}\label{RWAveraging}
Consider  $A\in C^\infty([0,1],\mathfrak{u}(n))$ and let $(B_{\epsilon})_{\epsilon>0}$ be a finite sum of elements in $S(\alpha)$
%$B_{\epsilon}(\tau)\in \mathfrak{u}(n) $ defined in Definition \ref{B} 
with $\alpha>1$.
Assume that $i A(\cdot)$ satisfies 
a $\kappa$-th order gap condition for some nonnegative integer $\kappa$.
%Assume that Condition (\ref{gap}) is satisfied.
Select $\lambda_j\in C^\infty([0,1],\R)$, $j=1,\dots,n$,  and $P\in C^\infty([0,1],{\rm U}(n))$  such that, for $j=1,\dots,n$ and $\tau\in [0,1]$, $\lambda_j(\tau)$ 
%(\mario{$i\lambda_j(\tau)$}) 
and the $j$-th column of $P(\tau)$ are, respectively,  an eigenvalue of $iA(\tau)$ %(\mario{$A(\tau)$}) 
and a corresponding eigenvector.
% (the  existence of $C^\infty$ eigenpairs being guaranteed by 
%Lemma~\ref{reg}).
Define
 $\Lambda(\tau)=\mathrm{diag}(\lambda_j(\tau))_{j=1}^n$, $\tau\in [0,1]$.
Fix $X_0\in \C^n$. % independent of $\epsilon$.
Let $X_{\epsilon}$ be the solution of $\frac{dX_{\epsilon}(s)}{ds}= (A(\epsilon s)+B_{\epsilon}(\epsilon s))X_{\epsilon}(s)$ such that $X_{\epsilon}(0)=X_{0}$.
Set $\Upsilon_{\epsilon}(\tau)=P(\tau) \exp\left(-\frac{i}{\epsilon} \int_0^{\tau} \Lambda(\vartheta) d\vartheta\right)  \exp\left(\int_{0}^{\tau} D(\vartheta)d\vartheta\right) P^{*}(0)$, where $D$ is equal to the diagonal part of $\frac{dP^{*}}{d\tau}P$. 
Then \[\|X_{\epsilon}(\tau/\epsilon) -\Upsilon_{\epsilon}(\tau) X_0\|<c \epsilon^{\min(\frac{1}{\kappa+1},\alpha-1)}\]
for some %constant 
$c>0$ independent of $\tau\in [0,1]$ and $\epsilon>0$.
\end{proposition}
\begin{proof}
Define $\Gamma(\tau)=\int_0^\tau \Lambda(s) ds$ and $Y_{\epsilon}(\tau)=\exp\left(\frac{i}{\epsilon} \Gamma(\tau) \right) P^{*}(\tau) X_{\epsilon} (\tau/\epsilon)$. 
Then $Y_{\epsilon}$ satisfies %the equation
\begin{equation}
\frac{dY_{\epsilon}(\tau)}{d\tau} =\left(D_{\epsilon}(\tau) +M(P,\Gamma,\epsilon)(\tau)\right)Y_{\epsilon}(\tau),\label{eqpert}
\end{equation}
%\begin{align}
%\frac{dY_{\epsilon}(\tau)}{d\tau} ={}&\Big[\exp\left(\frac{i}{\epsilon} \Gamma(\tau)\right) \frac{dP^{*}}{d\tau}(\tau) P(\tau) \exp\left(-\frac{i}{\epsilon}\Gamma(\tau) \right)%Y_{\epsilon}(\tau)
%\nonumber\\
%& +M(P,\Gamma,\epsilon)(\tau)\Big]Y_{\epsilon}(\tau),\label{eqpert}
%\end{align}
where $M(P,\Gamma,\epsilon)$ is defined as in~\eqref{def:MP} and  
%In order to simplify the notations, set 
\[ D_{\epsilon}(\tau)=\exp\left(\frac{i}{\epsilon} \Gamma(\tau)\right) \frac{dP^{*}}{d\tau}(\tau) P(\tau) \exp\left(-\frac{i}{\epsilon} \Gamma(\tau) \right).\]
Denote by $P_{\tau}^{\epsilon}$ and  $W_{\tau}^{\epsilon}$ the flows % at time $\tau$ 
of the equations 
$\frac{dx(\tau)}{d\tau}= M(P,\Gamma,\epsilon)(\tau)x(\tau)$ and $\frac{dx(\tau)}{d\tau}=(P_{\tau}^{\epsilon})^{-1} D_{\epsilon}(\tau) P_{\tau}^{\epsilon}x(\tau)$,
 respectively.
By the variations formula, % (Proposition~\ref{variation}), 
%we get that 
the flow %at time $\tau$ 
of  equation (\ref{eqpert}) is equal to $Q_{\tau}^{\epsilon} =P_{\tau}^{\epsilon}W_{\tau}^{\epsilon} $.
By Proposition~\ref{avnb} and Lemma~\ref{conjug}, we have $P_{\tau}^{\epsilon}=\mathrm{Id} +O(\epsilon^{\alpha-1})$.
Hence 
$(P_{\tau}^{\epsilon})^{-1} D_{\epsilon}(\tau)P_{\tau}^{\epsilon}=D_{\epsilon}(\tau)+O(\epsilon^{\alpha-1})$.

Using the 
$\kappa$-th order gap condition satisfied by $i A(\cdot)$, 
%gap condition (\ref{gap}), 
we have the estimate $\int_0^\tau D_{\epsilon}(\vartheta)d\vartheta=\int_0^\tau D(\vartheta) d\vartheta+O(\epsilon^{\frac{1}{\kappa+1}})$, uniformly with respect to $\tau \in [0,1]$.
Indeed, %the $(j,l)$-th coefficient  of $D_{\epsilon}(\tau)$ can be written as 
each coefficient of $D_{\epsilon}(\tau)$ can be written as 
$(D_{\epsilon}(\tau))_{jl}=q_{jl}(\tau)e^{\frac{i}{\epsilon}\int_0^\tau (\lambda_j(\vartheta)-\lambda_l(\vartheta))d\vartheta}$, %$j,l\in\{1,\dots,n\}$, 
where $q_{jl}$ is in $C^{\infty}([0,1],\R)$, and the conclusion follows by, e.g., \cite[Corollary A.6]{ABS19}.

% a direct estimation of the integral of the oscillating term $e^{\frac{i}{\epsilon}\int_0^\tau (\lambda_j(s)-\lambda_l(s))}ds$, $j,l\in\{1,\dots,n\}$.
Moreover, since $D_{\epsilon}$ is bounded with respect to $\epsilon$, 
we can conclude by standard averaging theory (see, e.g., \cite[Theorem A.1]{ABS19} for a closely related formulation)
that  
\[W_{\tau}^{\epsilon}=\exp \left(\int_0^\tau D(\vartheta) d\vartheta\right)+O(\epsilon^{\min(\frac{1}{\kappa+1},\alpha-1)}).\]
It follows that 
\begin{align*}
Q_{\tau}^{\epsilon}
 =&\left(\mathrm{Id}+O(\epsilon^{\alpha-1})\right)\\
 &\times \left( \exp \left(\int_0^\tau D(\vartheta) d\vartheta\right)+O(\epsilon^{\min(\frac{1}{\kappa+1},\alpha-1)})\right)\\
=&\exp \left(\int_0^\tau D(\vartheta) d\vartheta\right)+O(\epsilon^{\min(\frac{1}{\kappa+1},\alpha-1)}),
\end{align*}
concluding the proof of the proposition.
\end{proof}

\subsection{The ensemble case}\label{sec:ensemble}

%As detailed below,
Theorem~\ref{nlevel:thm:res} 
can be extended to 
the ensemble control setting in which a parametric dispersion 
affects the Hamiltonian $H_1$. 
The key argument allowing for such an extension is that the estimates obtained in Section~\ref{sec:Drive} can be made uniform with respect to the dispersion parameter. This is because the underlying averaging estimates (we refer in particular to \cite[Corollary A.6]{ABS19} used in the proof of Proposition~\ref{RWAveraging}) can be replaced by uniform parametric estimates such as those in \cite[Corollary A.8]{ABS19}.

Let $K$ be a compact subset of $\R^N$, for some $N\in \N$, and assume that $K\ni \delta\mapsto H_{1,\delta}\in i{\rm u}(n)$ is a continuous map. 
Let $H_0$ be as  in Section~\ref{sec:decou} and extend the definition of $\Xi$ by setting 
 \[\Xi=\left\{ |E_j-E_k|\mid 1\le j,k\le n,\,%\in  \{1,\dots,n\},\,
 \exists \delta\in K \mbox{\,s.t.\,}(H_{1,\delta})_{j,k}
\neq 0 \right\}.\]
% \[\Xi=\left\{\sigma\in [0,+\infty) \mid \left. \begin{aligned}
%  \exists (\delta,j,k)\in K \times \{1,\dots,n\}^2\mbox{ s.t. }\\
%  \sigma=|E_j-E_k|,\; 
%(H_{1,\delta})_{j,k}
%\neq 0\end{aligned} \right. \right\}
%.\]
With every $\delta\in K$ we can associate the 
decoupled Hamiltonians $H_{{\rm d},\delta}$ and $h_{{\rm d},\delta}$, and we notice that the control
 $u_{\epsilon}$ does not depend on $\delta$. (To be precise, one should assume condition 
 \eqref{eq:nonres-phi} to hold for every $(j,k),(p,q)$ such that 
 $E_{j}-E_{k}=E_p-E_q$ and $(H_{1,\delta})_{j,k}\ne 0\ne (H_{1,\delta})_{p,q}$ for some $\delta\in K$.) 
Then we have the following. 

\begin{theorem}\label{nlevel:thm:res-paramertric}
Set $\psi_0\in \C^n$ and let $\psi_{\epsilon}^\delta(\cdot)$ be the solution of
\begin{equation*}%\label{}
i\frac{d\psi_{\epsilon}^\delta(t)}{dt}=\left(H_0+u_{\epsilon}(t)H_{1,\delta}\right) \psi_{\epsilon}^\delta(t),\qquad 0\le t\le\frac1{\epsilon^{\alpha+1}},
\end{equation*} 
with $\psi_{\epsilon}^\delta(0)=\psi_0$.
% where by a slight abuse of notation, we write \[u_{\epsilon}(\tau)=2\sum_{(k,q)\in \Xi} v_{kq}(\tau)\cos(\frac{(E_k-E_q)\tau}{\epsilon}+\left(\varphi_k(\tau)-\varphi_q(\tau)\right)).\]
%
Define $\hat{\Psi}_{\epsilon}^\delta$ %:[0,1/\epsilon]\to \C^n$ 
as the solution of
\begin{equation}
\frac{d\hat{\Psi}_{\epsilon}^\delta(s)}{ds}=h_{\rm d,\delta}(\epsilon s) \hat{\Psi}_{\epsilon}^\delta(s),\qquad 0\le s\le \frac1\epsilon,
\end{equation} 
with $\hat{\Psi}_{\epsilon}^\delta(0)=\psi_0$.
Assume that for every $\delta\in K$ the non-autonomous Hamiltonian $h_{{\rm d},\delta}$ satisfies a $\kappa$-th order gap condition for some nonnegative integer $\kappa$ independent of $\delta$.
Then
 \[\left\|\psi_{\epsilon}^\delta\left(\frac{\tau}{\epsilon^{\alpha+1}}\right) -V_{\epsilon}%\left(\frac{s}{\epsilon}\right)
 (\tau)\hat{\Psi}_{\epsilon}^\delta\left(\frac{\tau}{\epsilon}\right)\right\|<c \epsilon^{\min(\frac1{\kappa+1},\alpha-1)}\] 
for $0\le \tau\le 1$, where
  $V_{\epsilon}(\tau)=\mathrm{diag}\left(e^{-i(\frac{E_{j}\tau}{\epsilon^{\alpha+1}}+\frac{\varphi_j(\tau)}{\epsilon})}\right)_{j=1}^n$
   and 
 $c>0$ is independent of $\tau\in [0,1]$, $\delta\in K$, and $\epsilon>0$.
  \end{theorem}

\section{Control strategies}\label{sec:strategies}

We present in this section some control strategies obtained by applying the general construction introduced in the previous sections, using different choices of the functional control parameters $v$ and $\varphi$ (cf.~\eqref{hatH}).% In order to simplify the notations, the strategies are described in the finite-dimensional case. Extensions to the infinite-dimensional case can be obtained in accordance with the results of Section~\ref{sec:infinite-dim}. 

\subsection{Decoupled Hamiltonian for a non-resonant coupling}\label{effective}
%\red{[IDEA: DIFFERENT REPARAMETRIZATIONS OF THE SYSTEM GIVE DIFFERENT NORMAL FORMS]}

Let $H_0$ and $H_1$ be as in Section~\ref{sec:decou}.

\begin{definition}
Given $j,l\in \{1,\dots,n\}$, we say that $H_1$
\emph{ non-resonantly couples the levels $j$ and $l$}
if there exist $j_1,\dots,j_k\in \{1,\dots,n\}$ such that 
%\begin{itemize}
%\item 
$j_1=j$, $j_k=l$, and, 
%\item 
for every $r=1,\dots,k-1$, $(H_1)_{j_r,j_{r+1}}\ne 0$ and $\mathcal{R}_{|E_{j_r}-E_{j_{r+1}}|}$ is equal either to $\{(j_r,j_{r+1})\}$ or to $\{(j_{r+1},j_{r})\}$ (where $\mathcal{R}$ is defined as in \eqref{Rsigma}).
%\end{itemize}
\end{definition}

Let us assume that $H_1$
non-resonantly couples two levels. Up to permutation, we can assume that $j=1$, $l=m$, and $j_k=k$ for $k=1,\dots,m$. 
Set $\sigma_j=|E_j-E_{j+1}|$ for $j=1,\dots,m-1$.
Notice that, in terms of the constraint \eqref{eq:nonres-phi} entering in the definition of the decoupling control 
$u_\epsilon$, the non-resonant coupling condition implies that the functions $\hat\varphi_{\sigma_j}$, $j=1,\dots,m-1$, can be chosen freely (up to the relation $\hat\varphi_{\sigma_j}(0)=0$).
Consider then  $\alpha>1$
 %$\hat \varphi_{\sigma_1},\dots,\hat \varphi_{\sigma_{m-1}}$ 
and  $v_1,\dots,v_{m-1}$ in $C^{\infty}([0,1],\R)$, and define 
\[
u_{\epsilon}(t)=
2\epsilon\sum_{j=1}^{m-1} 
\frac{v_{j}(\epsilon^{\alpha+1} t)}{(H_1)_{j,j+1}}
\cos\left(%\chi_{\sigma}(\tau)
|E_{j}-E_{j+1}|t+%\mario{{\rm sign}(E_j-E_{j+1})}
\frac{\hat \varphi_j(\epsilon^{\alpha+1}t)%-\varphi_{j+1}(\epsilon^{\alpha+1}t)
}{\epsilon}\right),
%\epsilon\sum_{\sigma\in \Xi}\xi_\sigma v_{\sigma}(\epsilon^{\alpha+1} t)\sum_{(j,k)\in \mathcal{R}_{\sigma}} 
%\frac{v_{[j,k]}}{r_{jk}}
%\cos\left((E_j-E_k)t+\frac{1}{\epsilon}(\varphi_j(\epsilon^{\alpha+1} t)-\varphi_k(\epsilon^{\alpha+1} t))\right),
%2\epsilon^{\alpha}\sum_{[j,k]\in \mathcal{R}} %\frac{v_{pq}}{r_{pq}}
%v_{kj}(\epsilon^{\alpha+1}t)\cos\left((E_j-E_k)t+\frac{1}{\epsilon}\left(\varphi_j(\epsilon^{\alpha+1}t)-\varphi_k(\epsilon^{\alpha+1}t)\right)\right),
%\qquad 0\le t\le \frac{1}{\epsilon^{\alpha+1}},
\]
for every $t\in [0, \frac{1}{\epsilon^{\alpha+1}}]$.
%where $v_1,\dots,v_{m-1}$ and $\varphi_1,\dots,\varphi_m$ are in $C^{\infty}([0,1],\R)$.
 %
%
%\[ u_{\epsilon}(\tau)=\frac{2}{\epsilon} \sum_{(p,q)\in \tilde{\Xi}} \frac{v_{pq}(\tau)}{(H_1)_{pq}}\cos((E_p-E_q)\frac{\tau}{\epsilon^{\alpha+1}}+\frac{1}{\epsilon}\left(\varphi_p(\tau)-\varphi_q(\tau)\right))\]
% where, for every $(p,q)\in \tilde{\Xi}$, $v_{pq}\in C^{\infty}([0,1],\R)$, and $\varphi_p\in C^{\infty}([0,1],\R)$ for every $p\in \tilde{S}$.
For simplicity of notations, let  $v=(v_{j})_{j=1}^{m-1}$ and $\varphi=(\varphi_{j})_{j=1}^{m}$,
where $\varphi_1,\dots,\varphi_m\in C^\infty([0,1],\R)$ satisfy
${\rm sign}(E_j-E_{j+1})(\varphi_j-\varphi_{j+1})=\hat\varphi_j$.

In analogy with \eqref{H-decoupled} and \eqref{hatH}, define $H_{\rm d}:\R^{m}\times \R^{m-1}\to i\mathfrak{u}(n)$ and $h_{\rm d}:[0,1]\to i\mathfrak{u}(n)$ by
\[H_{\rm d}(\delta,w)={\scriptsize\begin{pmatrix}
\delta_1&w_{1}&0&\cdots&\cdots&\cdots&0\\
w_{1}&\delta_{2}&w_{2}&\ddots&\ddots&\ddots&\vdots\\
0&w_{2}&\delta_{3}&w_{3}&\ddots&\ddots&\vdots\\
\vdots&\ddots&\ddots&\ddots&\ddots&\ddots&\vdots\\
\vdots&\ddots&\ddots&w_{m-2}&\delta_{m-1}&w_{m-1}&0\\
\vdots&\ddots&\ddots&\ddots&w_{m-1}&\delta_{m}&0\\
0&\cdots&\cdots&\cdots&0&0&0
\end{pmatrix}},\] and  $h_{\rm d}(\tau)=H_{\rm d}(-\varphi'(\tau),v(\tau))$,
where the zeros in the last line and in the last column of $H_{\rm d}(\delta,w)$ are null matrices of suitable dimensions.
Notice that, with respect to the notations of Section~\ref{sec:decou}, we are setting here $v_0\equiv0$.

\subsection{Decoupled multilevel chirp pulse}

Consider 
$u,\phi\in C^\infty([0,1],\R)$ 
to be chosen later. 
Assume that $v_{j}=u$ for every $j\in \{1,\dots,m-1\}$ and that 
$\varphi_{j}=j \phi$  for $j\in \{1,\dots,m\}$.
Hence,
$h_{\rm d}(\tau)=H_{\rm C}(-\varphi'(\tau),u(\tau))$, where
\[H_{\rm C}(\rho,w)=
{\scriptsize\begin{pmatrix}
\rho&w&0&\cdots&\cdots&\cdots&0\\
w&2 \rho&w&\ddots&\ddots&\ddots&\vdots\\
0&w&3 \rho&w&\ddots&\ddots&\vdots\\
\vdots&\ddots&\ddots&\ddots&\ddots&\ddots&\vdots\\
\vdots&\ddots&\ddots&w&(m-1)\rho&w&0\\
\vdots&\ddots&\ddots&\ddots&w&m\rho&0\\
0&\cdots&\cdots&\cdots&0&0&0
\end{pmatrix}}.\]

Recall the following result of linear algebra, that can be found, for instance, in~\cite{RouchonSarlette}.
\begin{lemma}\label{simple:eig}
Let $A$ be the real-valued tridiagonal $n\times n$ matrix 
 \[A={\scriptsize\begin{pmatrix}
a_1&c_1&0&0&0&\dots\\
c_1&a_2&c_2&0&0&\dots\\
0&c_2&a_3&c_3&0&\dots\\
0&0&c_3&a_4&c_4&\dots\\
\vdots&\vdots&\vdots&\vdots&\vdots&a_n\\
\end{pmatrix}}.\]
If $c_k\neq 0$ for every $k\in \{1,\dots,n-1\}$, then the eigenvalues of $A$ are simple.
\end{lemma}

%{\color{blue}
%
%\begin{cor}[SUPER-CHIRP FORM]\label{superchirp}
%Assume that $v_{\sigma^{k}(j),\sigma^{k+1}(j)}=u(\tau)$ for every $k\in \{1,\dots, |\tilde{S}|\}$, and $\varphi_{\sigma^{k-1}(j)}(\tau)=k \varphi(\tau)$, where $u,\varphi \in C^{\infty}([0,1],\R)$, that is \[ u_{\epsilon}(\tau)=\frac{2}{\epsilon} \sum_{(p,q)\in \tilde{\Xi}} \frac{u(\tau)}{(H_1)_{pq}}\cos((E_p-E_q)\frac{\tau}{\epsilon^{\alpha+1}}+\frac{1}{\epsilon}\varphi(\tau)).\]
%Then
% \[h(\varphi',u)=\begin{pmatrix}
%\varphi'&u&0&0&0&\dots\\
%u&2\varphi'&u&0&0&\dots\\
%0&u&3\varphi'&u&0&\dots\\
%0&0&u&4\varphi'&u&\dots\\
%\vdots&\vdots&\vdots&\vdots&\vdots&d\varphi'\\
%\end{pmatrix}.\]
%
%\end{cor}
%}
%
% Assume that \[h(\varphi',u)=\begin{pmatrix}
%\varphi'&u&0&0&0&\dots\\
%u&2\varphi'&u&0&0&\dots\\
%0&u&3\varphi'&u&0&\dots\\
%0&0&u&4\varphi'&u&\dots\\
%\vdots&\vdots&\vdots&\vdots&\vdots&d\varphi'\\
%\end{pmatrix}.\]
%\red{Contrarly to the SUPER STIRAP, this model is not controllable when the dimension is greater than $3$.}
It follows from Lemma~\ref{simple:eig} that the only eigenvalue intersection  of 
top-left $m\times m$ submatrix of $H_C(\rho,w)$
is located at $\rho=w=0$.

%\red{[TRANSITION FROM THE LOWER TO THE UPPER STATE]}

Denote by $e_1,\dots,e_n$ the canonical basis of $\C^n$. 
By Theorem~\ref{nlevel:thm:res}, we have the following result.
\begin{proposition}\label{SUPERCHIRP}
Let $u,\phi\in C^{\infty}([0,1],\R)$ be such that $u(0)=u(1)=0$, 
 $\phi(0)=0$, $\phi'(0)\phi'(1)<0$,  and  
 $u(\tau)\neq 0$ for $\tau \in (0,1)$. 
For every $\epsilon>0$, let  $\psi_\epsilon:[0,\frac{1}{\epsilon^{\alpha+1}}]\to\C^n$ be the solution of \eqref{Eq:gen} 
with initial condition $e_1$
associated with the control $u_\epsilon$. 
Then, there exists $C>0$ independent of $\epsilon$ such that $\|{\psi}_\epsilon(\frac{1}{\epsilon^{\alpha+1}})-e^{i\theta_{\epsilon}}e_m\|\leq C\epsilon^{\min(1,\alpha-1)}$ for some $\theta_{\epsilon}\in \R$.
%Then, if $\psi_0=e_j$, then $\|\hat{\psi}_\epsilon(1)-e^{i\theta_{\epsilon}}e_l\|\leq C\epsilon^{\frac{1}{2}}$, where $C>0$ is independent of $\epsilon$, and $\theta_{\epsilon}\in \R$.
\end{proposition}

\subsubsection{Simulations}\label{sec:sim1}
Consider, for $\tau\in [0,1]$, $\phi(\tau)=-\frac{2}{\pi} \sin(\pi \tau)$ and $u(\tau)=4\sin(\pi \tau)$, and define $H_0=\text{diag}(E_j)_{j=1}^7$. Let $H_1$ be the $7\times 7$ symmetric matrix such that for every $1\le j\le k\le 7$, $(H_1)_{j,k}=1$ if $k=j+1$ and $0$ otherwise. 
%0&1&0&0&0&0&0\\
%1&0&1&0&0&0&0\\
%0&1&0&1&0&0&0\\
%0&0&1&0&1&0&0\\
%0&0&0&1&0&1&0\\
%0&0&0&0&1&0&1\\
%0&0&0&0&0&1&0\\
%\end{pmatrix}$. 
Let $\psi_\epsilon:[0,\frac{1}{\epsilon^{\alpha+1}}]\to\C^n$ be the solution of \eqref{Eq:gen} 
with initial condition $e_1$ associated with the control $u_{\epsilon}$ as in Proposition~\ref{SUPERCHIRP}.
Define for $j\in \{1,\dots,7\}$, the \emph{population in level $j$} as $p_j(\tau)=| \langle \psi_{\epsilon}(\frac{\tau}{\epsilon^{\alpha+1}}), e_j \rangle |^2$ for $\tau\in [0,1]$.

We have plotted on Figure~\ref{sub1} the 
population levels of $\psi_\epsilon$
%solution of Equation~\eqref{Eq:gen} associated with the control $u_{\epsilon}$ 
in the case $\alpha=1.2$, with $E_1=0, \ E_2=1, \ E_3 = 2.5, \ E_4 = 3, \ E_5 = 2.2, \ E_6 = 5, \ E_7 = 7$. %, and for $\tau\in [0,1]$.
The case where $\alpha=0.8$ (violating the hypothesis $\alpha>1$), while the other parameters are the same, is illustrated in  Figure~\ref{sub2}.
%We have plotted on Figure~\ref{sub2} the solution of Equation~\eqref{Eq:gen} associated with the control $u_{\epsilon}$ with $\alpha=0.8$, while the other parameters are the same as previously.

%\begin{figure}[h!]
%\center
%\includegraphics[scale=1]{sc3.pdf}
%\caption{Monochromatic case $1-2$-transition: Evolution of the different populations as a function of the renormalized time $\tau\in [0,1]$, $\alpha=0.8$ and $\epsilon=10^{-2}$.}
%\label{scfail}
%\end{figure} 
\begin{figure}[h!]
\begin{center}
    \subfigure[$\alpha=1.2$]{\label{sub1} \includegraphics[scale=0.6]{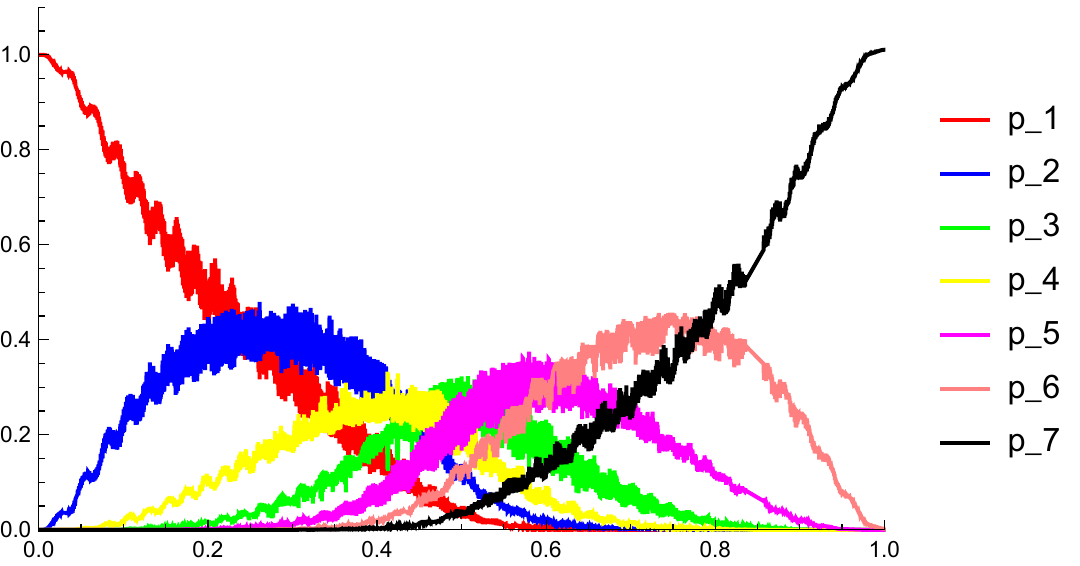}}
    \subfigure[$\alpha=0.8$]{\label{sub2} \includegraphics[scale=0.6]{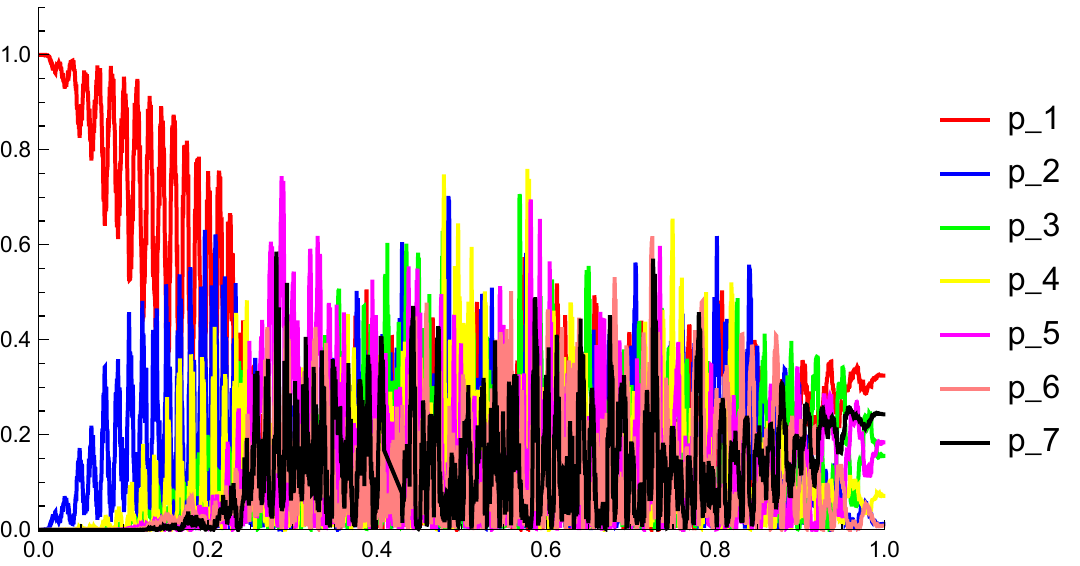}}
    \caption{Evolution of the  populations $(p_j(\tau))_{j=1}^7$ as functions of the renormalized time $\tau=\epsilon^{\alpha+1} t$, % \in [0,1]$,
  for   $\epsilon=10^{-2}$ and  $\alpha=1.2$ (\ref{sub1}), $\alpha=0.8$ (\ref{sub2}).}
    \label{sup:chirp}
    \end{center}
\end{figure}

\begin{remark}\label{rem:robust-superchirp}
Applying Theorem~\ref{nlevel:thm:res-paramertric} instead of Theorem~\ref{nlevel:thm:res}, we can recover a version of Proposition~\ref{SUPERCHIRP} allowing for parametric dispersion in the controlled Hamiltonian $H_1$ of the form $H_{1,\delta}=\delta H_1$, with $\delta\in [\delta_0,\delta_1]$ for some $0<\delta_0<\delta_1$. Indeed, 
the property of non-resonantly coupling two levels is independent of $\delta$ and 
$u_\epsilon$ depends on $\delta$  only through a positive constant multiplicative factor. Hence the control $u$ in Proposition~\ref{SUPERCHIRP} is also modified 
through a positive constant multiplicative factor, while $\phi$ does not depend on $\delta$. 
%The hypotheses of Proposition~\ref{SUPERCHIRP}  for every $\delta\in [\delta_0,\delta_1]$
%
Proposition~\ref{SUPERCHIRP} provides then an explicit robust control strategy 
reflecting the ensemble controllability result 
obtained by Chambrion in \cite[Proposition 1]{TENS}, where it is mentioned %in Conclusion 
that %the 
an ensemble
control strategy is difficult to implement %for real systems 
because of the ``poor efficiency of tracking strategies via Lie brackets". %{\color{red}CITE - with reference to statement}. 
\end{remark}
%\red{[FIGURE ROBUSTNESS: MAY BE TOO LONG??]}
%\subsection{Robustness of the control strategies}
%
%Using uniform estimates with respect to $\delta$ we get the following ensemble controllability result.
%
%\begin{theorem}[\red{ALREADY PROVED BY THOMAS CHAMBRION BY UNIFORM LIE METHODS}]
%Let $(H_0,H_1)$ satisfying the controllability criterion of Theorem~\ref{cont:nr}.
%Then the equation \begin{equation}
%i\frac{d\psi}{dt}=\left(H_0+\delta u H_1\right) \psi
%\end{equation} 
%is approximately ensemble controllable between the eigenstates of $H_0$, uniformly with respect to $\delta\in [a,b]\subset ]0,+\infty[$.
%
%\end{theorem}

\subsection{Decoupled multilevel STIRAP}
%\red{[TODO: CHANGE THE AXIS ON THE FIGURES, ARRANGER UN PEU]}
%\red{[SIMU ROBUST+SPLITTINGDIM3: TOO LONG??]}
	
%For other control purposes, such as to induce superposition of states or 
In order
to reduce the populations in the intermediate levels 
along the controlled motion
%during the control strategy 
(see, for instance \cite{SHORE91}), 
it is interesting to introduce another control strategy,
%one could be interested in the following effective dynamics, 
which generalizes the well-known Stimulated Raman Adiabatic Passage (STIRAP). 
We are going to see that the proposed %control 
strategy is different depending on the parity of the integer $m$ defined in Section~\ref{effective}.
% is odd or even.
%\footnote{\mario{PROPOSAL: we could say that we can always split in a odd plus eventually a two-level, which is solved as in the previous section, and just do the odd case}}

Let $(d_j)_{j=1}^m\subset \R$ be increasing %nondecreasing 
and consider $u_1,u_2\in C^\infty([0,1],\R)$ 
to be chosen later. 
%be such that
%$u_1(0)=u_1(1)=0$, $u_1(\tau)>0$ for $\tau\in (0,1)$, and $u_2(0)<d_1\le d_m<u_2(1)$.
For $j\in \{1,\dots,m-1\}$, let $v_{j}=u_1$ if $j$ is odd, and $v_{j}=u_2$ if $j$  is even.
By choosing $\varphi_{j}(\tau)=-d_j \tau$ for $j\in \{1,\dots,m\}$ and $\tau\in [0,1]$, we have that
$h_{\rm d}(\tau)=H_{\rm S}(u_1(\tau),u_2(\tau))$, where
\begin{equation}\label{eq:*u}
H_{\rm S}(w_1,w_2)=
{\scriptsize\begin{pmatrix}
d_1&w_{1}&0&\cdots&\cdots&\cdots&0\\
w_{1}&d_{2}&w_{2}&\ddots&\ddots&\ddots&\vdots\\
0&w_{2}&d_{3}&w_{1}&\ddots&\ddots&\vdots\\
\vdots&\ddots&\ddots&\ddots&\ddots&\ddots&\vdots\\
\vdots&\ddots&\ddots&w_{1}&d_{m-1}&w_2&0\\
\vdots&\ddots&\ddots&\ddots&w_2&d_{m}&0\\
0&\cdots&\cdots&\cdots&0&0&0
\end{pmatrix}}.
\end{equation}
(The expression of $H_{\rm S}$ in \eqref{eq:*u} corresponds to the case where $m$ is odd, the roles of $w_1$ and $w_2$ in the last lines of the matrix being inverted if $m$ is even). 
%
%
%\begin{pmatrix}
%d_1&u&0&0&0&\dots\\
%u&d_{2}&v&0&0&\dots\\
%0&v&d_{\sigma^2(j)}&u&0&\dots\\
%0&0&u&d_{\sigma^3(j)}&v&\dots\\
%\vdots&\vdots&\vdots&\vdots&\vdots&d_{\sigma^{d-1}(j)}\\
%\end{pmatrix}.\]

%
%{\color{blue}
%
%\begin{cor}[SUPER-STIRAP FORM]\label{superstir}
%Let $(d_k)_{k\in \{1,\dots, |\tilde{S}|\}}$ be increasing.
%Assume that $v_{\sigma^{k}(j),\sigma^{k+1}(j)}=u(\tau)$ if $k\in \{1,\dots, |\tilde{S}|\}$ is even, and $v_{\sigma^{k}(j),\sigma^{k+1}(j)}=v(\tau)$ if $k\in \{1,\dots, |\tilde{S}|\}$ is odd, and $\varphi_{\sigma^k(j)}(\tau)=d_k \tau$.
%Then \[h(u,v)=\begin{pmatrix}
%d_j&u&0&0&0&\dots\\
%u&d_{\sigma(j)}&v&0&0&\dots\\
%0&v&d_{\sigma^2(j)}&u&0&\dots\\
%0&0&u&d_{\sigma^3(j)}&v&\dots\\
%\vdots&\vdots&\vdots&\vdots&\vdots&d_{\sigma^{d-1}(j)}\\
%\end{pmatrix}.\]
%\end{cor}
%}
%
%Assume that \[h(u,v)=\begin{pmatrix}
%d_j&u&0&0&0&\dots\\
%u&d_{\sigma(j)}&v&0&0&\dots\\
%0&v&d_{\sigma^2(j)}&u&0&\dots\\
%0&0&u&d_{\sigma^3(j)}&v&\dots\\
%\vdots&\vdots&\vdots&\vdots&\vdots&d_{\sigma^{d-1}(j)}\\
%\end{pmatrix},\] where $(d_k)_{k\in \{1,\dots,n\}}$ is increasing.
%It follows from Lemma~\ref{simple:eig} that intersections of eigenvalue of $h(u,v)$ are located on the axis $u=0$ and $v=0$.
%

Denote by $\lambda_1(w_1,w_2)\le \dots \le \lambda_m(w_1,w_2)$ the eigenvalues of the 
top-left $m\times m$ submatrix of $H_S(w_1,w_2)$.
As a consequence of Lemma~\ref{simple:eig}, the only
eigenvalue intersections between $\lambda_j$ and $\lambda_{j+1}$, $j\in\{1,\dots,m-1\}$,
are located either on the axis $w_1=0$ or on the axis $w_2=0$.

Depending on the parity of $m$, one of the following two lemmas can be applied. 
The lemmas can be deduced from \cite[Section II]{adami-boscain}.
\begin{lemma}[$m$ odd]\label{eig:int:odd}
Assume that $m$ is odd.
Then there exist two finite positive sequences  $(w_{2,k})_{k=1}^{\frac{m-1}{2}}$ and $(w_{1,k})_{k=\frac{m+1}{2}}^{m-1}$, which are respectively increasing and decreasing, such that\\
%\begin{itemize}
%\item
$\bullet$ for  $k\in \{1,\dots,\frac{m-1}{2}\}$, 
$w_2\mapsto \lambda_{k}(0,w_2)$ and $w_2\mapsto \lambda_{k+1}(0,w_2)$ have a transverse intersection at $w_{2,k}$;\\
%$(0,w_{2,k})$ is a transverse intersection in the direction $e_2$ between $\lambda_{k}$ and $\lambda_{k+1}$;
%\item 
$\bullet$ for  $k\in \{\frac{m+1}{2},\dots,m-1\}$, 
$w_1\mapsto \lambda_{k}(w_1,0)$ and $w_1\mapsto \lambda_{k+1}(w_1,0)$ have a transverse intersection at $w_{1,k}$.\\
%$(w_{1,k},0)$ is a transverse intersection in the direction $e_1$ between $\lambda_{k}$ and $\lambda_{k+1}$.
%\end{itemize}
Moreover, $\lambda_{k+1}(0,w_2)$ does not intersect $\lambda_k(0,w_2)$ nor $\lambda_{k+2}(0,w_2)$ for $w_2\in (w_{2,k}, w_{2,k+1})$, and $\lambda_{k+1}(w_1,0)$ does not intersect $\lambda_k(w_1,0)$ nor $\lambda_{k+2}(w_1,0)$ for $w_1\in (w_{1,k+1}, w_{1,k})$.
\end{lemma}

%\begin{lem}[$m$ odd]\label{eig:int:odd}
%Assume that $m$ is odd.
%Then there exist two finite positive sequences $(w_{2,k})_{k=1}^{\frac{m-3}{2}}$, $(w_{1,k})_{k=\frac{m-1}{2}}^m$, which are respectively increasing and decreasing, such that
%\begin{itemize}
%\item for every $k\in \{1,\dots,\frac{m-3}{2}\}$, $(0,w_{2,k})$ is a transverse intersection in the direction $e_2$ between $\lambda_{k}$ and $\lambda_{k+1}$;
%\item for every $k\in \{\frac{m-1}{2},\dots,m-1\}$, $(w_{1,k},0)$ is a transverse intersection in the direction $e_1$ between $\lambda_{k}$ and $\lambda_{k+1}$.
%\end{itemize}
%Moreover, $\lambda_{k+1}(0,w_2)$ does not intersect $\lambda_k(0,w_2)$ nor $\lambda_{k+2}(0,w_2)$ for $w_2\in (w_{2,k}, w_{2,k+1})$, and $\lambda_{k+1}(w_1,0)$ does not intersect $\lambda_k(w_1,0)$ nor $\lambda_{k+2}(w_1,0)$ for $w_1\in (w_{1,k+1}, w_{1,k})$.
%\end{lem}

\begin{lemma}[$m$ even]\label{eig:int:ev}
Assume that $m$ is even. 
Then there exist two finite positive sequences $(w_{2,k})_{k=1}^{\frac{m}{2}-1}$, $(w_{1,k})_{k=\frac{m}{2}}^{m-2}$, which are, respectively, increasing and decreasing, such that\\
%\begin{itemize}
%\item 
$\bullet$  for  $k\in \{1,\dots,\frac{m}{2}-1\}$, 
$w_2\mapsto \lambda_{k}(0,w_2)$ and $w_2\mapsto \lambda_{k+1}(0,w_2)$ have a transverse intersection at $w_{2,k}$;\\
%$(0,w_{2,k})$ is a transverse intersection in the direction $e_2$ between $\lambda_{k}$ and $\lambda_{k+1}$;
%\item 
$\bullet$ for  $k\in \{\frac{m}{2},\dots,m-2\}$, 
$w_1\mapsto \lambda_{k}(w_1,0)$ and $w_1\mapsto \lambda_{k+1}(w_1,0)$ have a transverse intersection at $w_{1,k}$.\\
%$(w_{1,k},0)$ is a transverse intersection in the direction $e_1$ between $\lambda_{k}$ and $\lambda_{k+1}$.
%\end{itemize}
Moreover, $\lambda_{k+1}(0,w_2)$ does not intersect $\lambda_k(0,w_2)$ nor $\lambda_{k+2}(0,w_2)$ for $w_2\in (w_{2,k}, w_{2,k+1})$, and $\lambda_{k+1}(w_1,0)$ does not intersect $\lambda_k(w_1,0)$ nor $\lambda_{k+2}(w_1,0)$ for $w_1\in (w_{1,k+1}, w_{1,k})$.
Furthermore, there exists $w_{2}^\star>0$ such that
$w_2\mapsto \lambda_{m-1}(0,w_2)$ and $w_2\mapsto \lambda_{m}(0,w_2)$ have a transverse intersection at $w_{2}^\star$ and \\
%there exists a transverse intersection $(0,w_2^{\star})$ in the direction $e_2$ between $\lambda_{m-1}$ and $\lambda_{m}$ such that 
%\begin{itemize}
%\item 
$\bullet$  $\lambda_{m-1}(0,w_2)$ does not intersect $\lambda_{m-2}(0,w_2)$ nor $\lambda_{m}(0,w_2)$ for $w_2\in (0,w_2^{\star})$;\\
%\item 
$\bullet$ $\lambda_{m}(0,w_2)$ does not intersect $\lambda_{m-1}(0,w_2)$ for $w_2\in (w_2^{\star},+\infty)$.\\
%\end{itemize}
Finally, $\lambda_{m}(w_1,0)$ does not intersect $\lambda_{m-1}(w_1,0)$ for $w_1\in [0,+\infty)$.
\end{lemma}

In order to illustrate Lemmas \ref{eig:int:odd} and \ref{eig:int:ev}, we have plotted in Figures \ref{5lev} and \ref{6lev} the spectrum of $H_{\rm S}$ for $m=5$ and $m=6$ on the axes $w_1=0$ and $w_2=0$.

\begin{figure}[h!]
    \centering
    \subfigure[The eigenvalues of $H_{\rm S}(0,w_2)$ for $m=5$.]{\label{sub12} \includegraphics[scale=0.25]{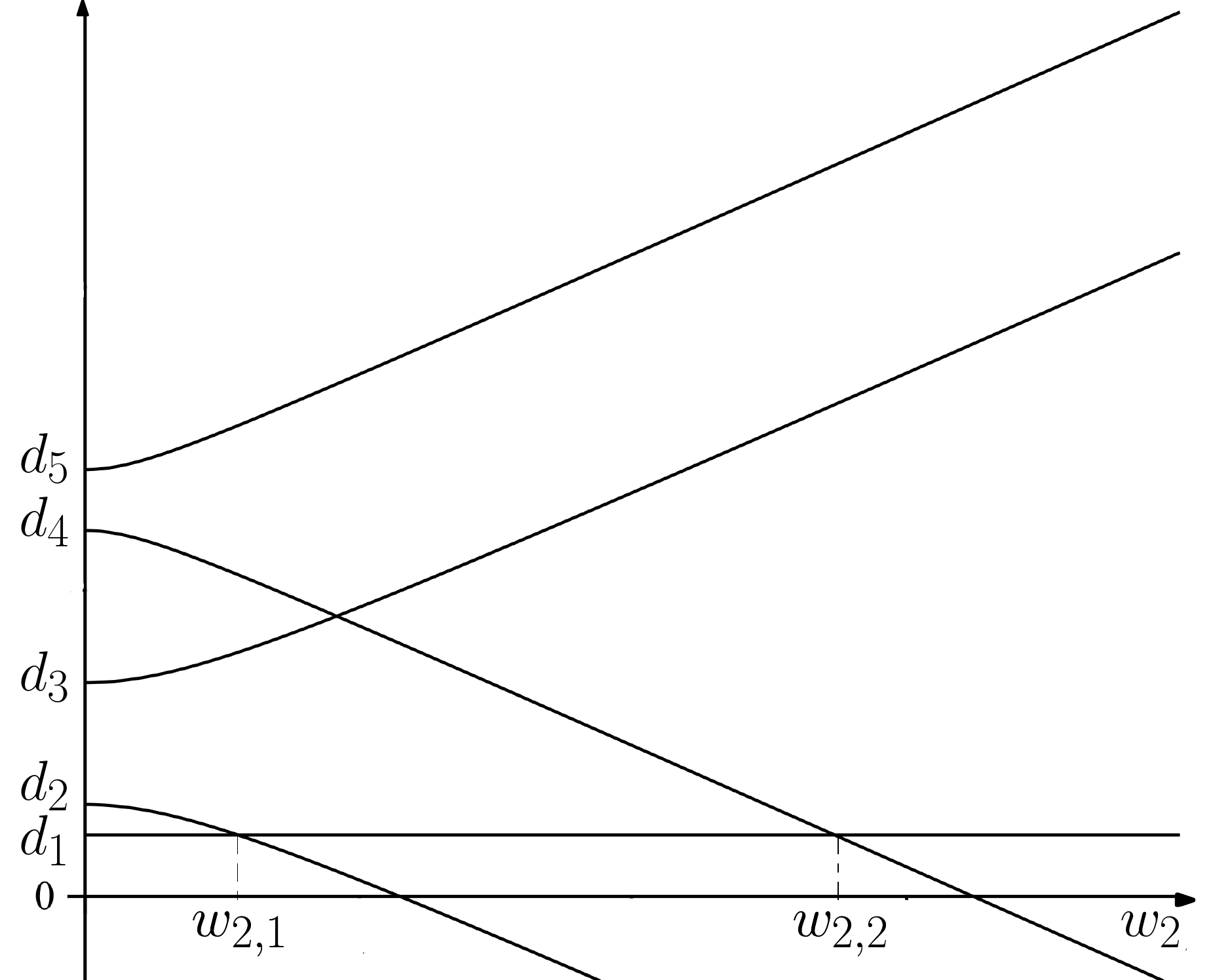}}
    \subfigure[The eigenvalues of $H_{\rm S}(w_1,0)$ for $m=5$.]{\label{sub22} \includegraphics[scale=0.25]{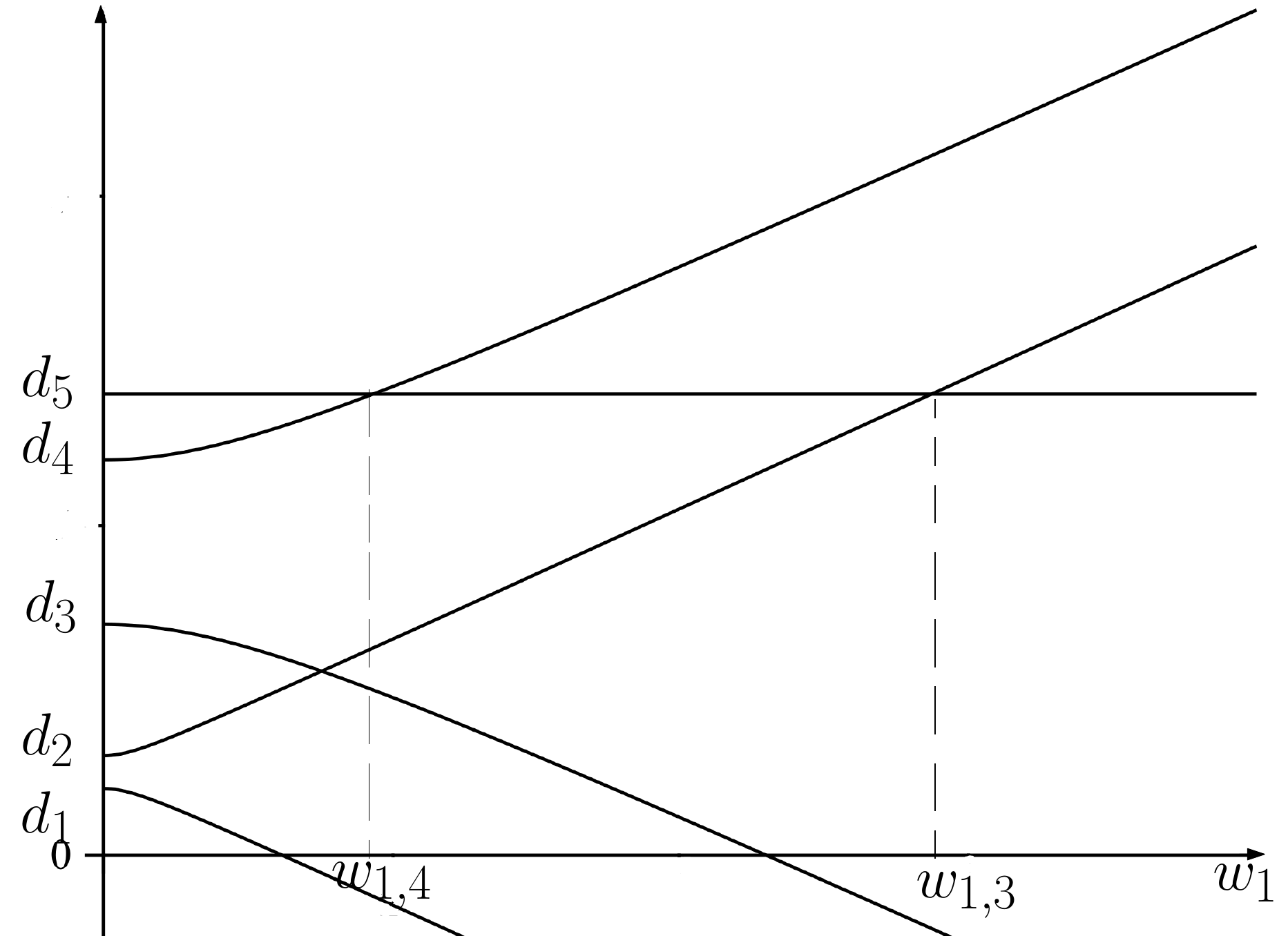}}
    \caption{ \label{5lev}}
\end{figure}

\begin{figure}[h!]
    \centering
    \subfigure[The eigenvalues of $H_{\rm S}(0,w_2)$ for $m=6$.]{\label{sub13} \includegraphics[scale=0.25]{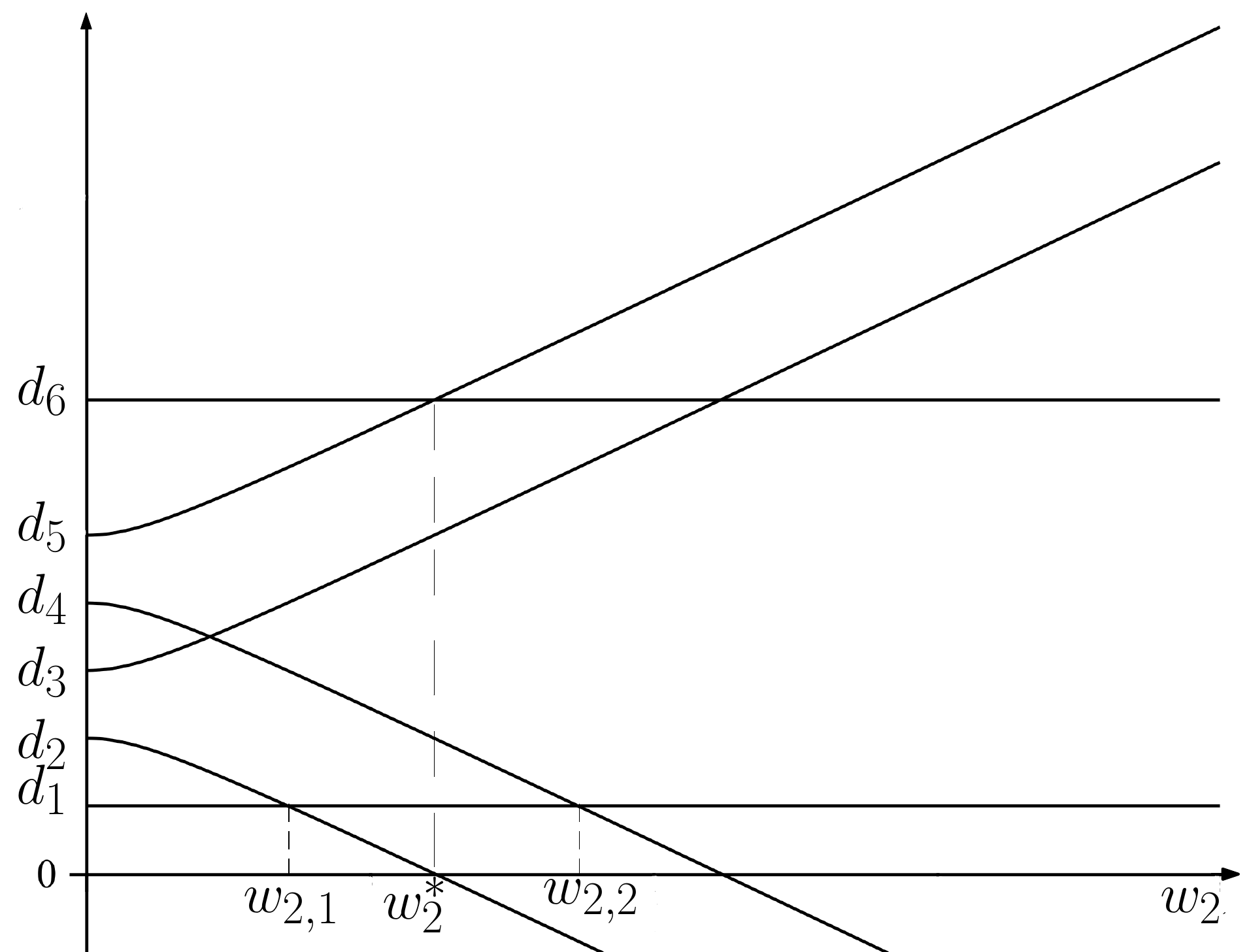}}
    \subfigure[The eigenvalues of $H_{\rm S}(w_1,0)$ for $m=6$.]{\label{sub23} \includegraphics[scale=0.25]{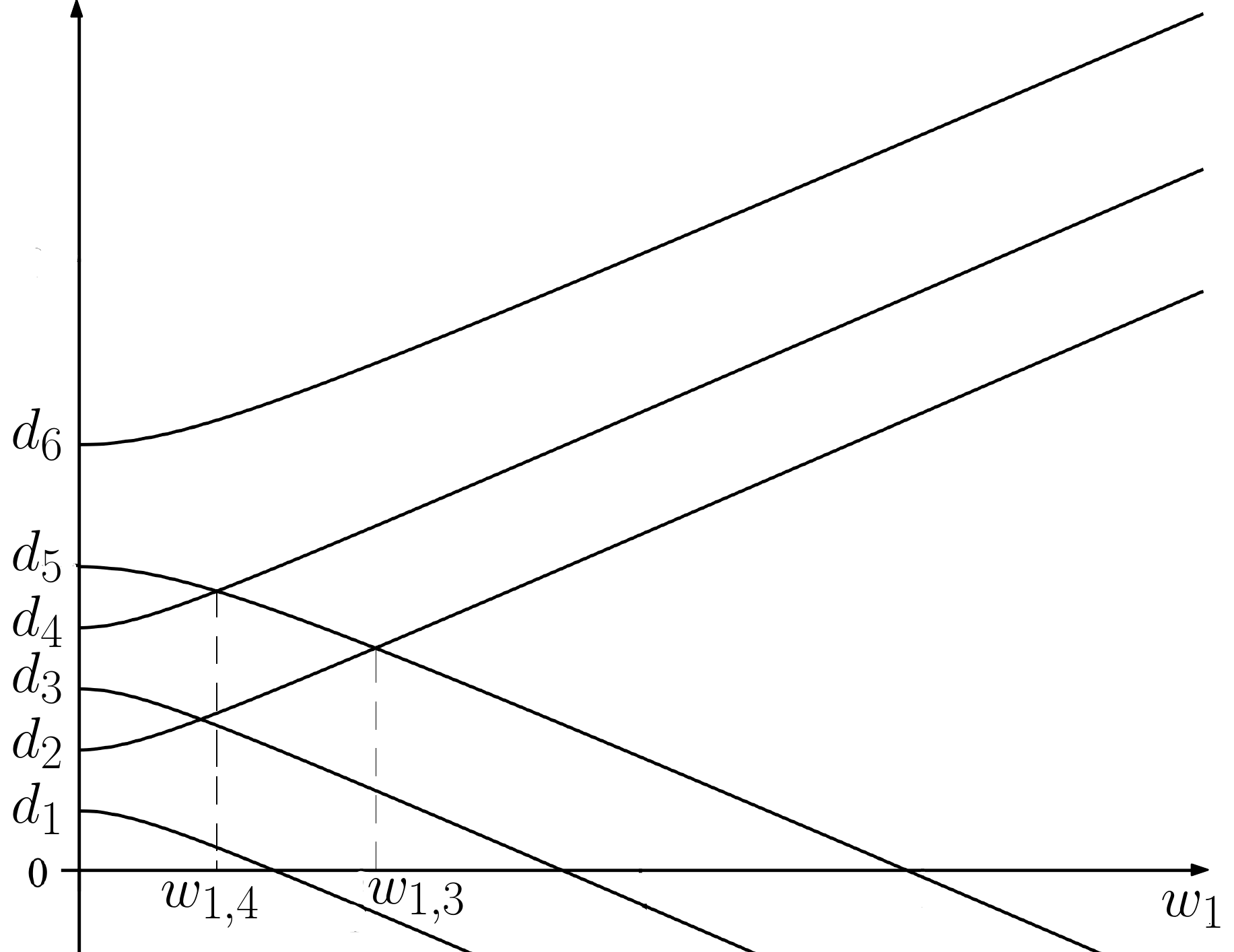}}
    \caption{\label{6lev}}
\end{figure}

\begin{remark}
%\red{[ONLY QUALITATIVE, OTHERWISE WE HAVE TO DO TOO MUCH TECHNICAL COMPUTATIONS]}
The eigenvalue intersections $\lambda_k(w_1,w_2)=\lambda_{k+1}(w_1,w_2)$
described in Lemmas \ref{eig:int:odd} and \ref{eig:int:ev} are not necessarily conical, 
even if they all are transverse 
in the directions that we are interested in (the horizontal or the vertical axis of the plane $(w_1,w_2)$). 
%($e_2$ or $e_1$ depending on the case $w_1=0$ or $w_2=0$). 
It is interesting to notice that numerical simulations show the presence of both conical and semi-conical eigenvalue intersections, using the terminology of \cite{ABS19,boscain_approximate_2015,Bos}.  
%the presence of semi-conical intersections, which were introduced by the authors in~\cite{ABS19}. 
%The latter intersections are more degenerate than conical intersections in the sense of the Whitney topology. They possess a direction (called \emph{non-conical direction})  along which the eigenvalues are tangent and the eigenvalues are transverse along every other directions.
%On Figure~\ref{semiconique}, we have plotted two eigenvalues $(w_1,w_2)\mapsto \left(\lambda_j(w_1,w_2),\lambda_{j+1}(w_1,w_2)\right)$ of $H_{\rm S}$ for a given $j\in \{2,\dots, m-2\}$. We notice the presence of two conical and two semi-conical eigenvalue intersections.
%On Figure~\ref{conique}, we have plotted the two eigenvalues $(w_1,w_2)\mapsto \left(\lambda_1(w_1,w_2),\lambda_{2}(w_1,w_2)\right)$ of $H_{\rm S}$. We notice the presence of two conical eigenvalues intersections.
%Up to a permutation of the coordinates $w_1$ and $w_2$ and $z\mapsto -z$, the spectrum of $H_{\rm S}$ can be decomposed as a superposition of these two basic motifs.
%
\end{remark}

%\begin{figure}[h!]
%    \centering
%    \subfigure[General shape of $(w_1,w_2)\mapsto \left(\lambda_j(w_1,w_2),\lambda_{j+1}(w_1,w_2)\right)$ for a given $j\in \{2,\dots, m-2\}$.]{\label{semiconique} \includegraphics[scale=0.25]{basique.eps}}
%    \subfigure[General shape of $(w_1,w_2)\mapsto \left(\lambda_1(w_1,w_2),\lambda_{2}(w_1,w_2)\right)$.]{\label{conique} \includegraphics[scale=0.25]{12.eps}}
%    \caption{\label{6lev-other}}
%\end{figure}

We deduce from Theorem~\ref{nlevel:thm:res} and Lemma~\ref{eig:int:odd} %
the following proposition.
%, \ref{eig:int:ev}  the two propositions below. 
%following results.
\begin{proposition}[$m$ odd]\label{p:ssodd}
Let $(u_1,u_2)$ satisfy the following properties: there exist $0<\tau_1<\tau_2<1$ such that $u_1|_{[0,\tau_1]}\equiv 0$, $u_2|_{[\tau_2,1]}\equiv 0$, $u_1(\tau),u_2(\tau)>0$ for $\tau\in (\tau_1,\tau_2)$, 
$u_2$ is increasing on $[0,\tau_1]$ from $0$ to a value larger than $w_{2,\frac{m-1}{2}}$, $u_1$ is decreasing on $[\tau_2,1]$ from a value larger than $w_{1,\frac{m+1}{2}}$ to $0$ (see Figure~\ref{loopodd}).
For every $\epsilon>0$, let 
$\psi_\epsilon:[0,\frac{1}{\epsilon^{\alpha+1}}]\to\C^n$ be the solution of \eqref{Eq:gen} 
with initial condition $e_1$
associated with the control $u_\epsilon$. 
Then, there exists $C>0$ independent of $\epsilon$ such that $\|{\psi}_\epsilon(\frac{1}{\epsilon^{\alpha+1}})-e^{i\theta_{\epsilon}}e_m\|\leq C\epsilon^{\min(\frac1{2},\alpha-1)}$ for some $\theta_{\epsilon}\in \R$.
\end{proposition}

\begin{figure}[h!]
    \centering
    \subfigure[Control path $(u_1,u_2)$ in the plane $(w_1,w_2)$ for $m$ odd.]{\label{loopodd} \includegraphics[scale=0.43]{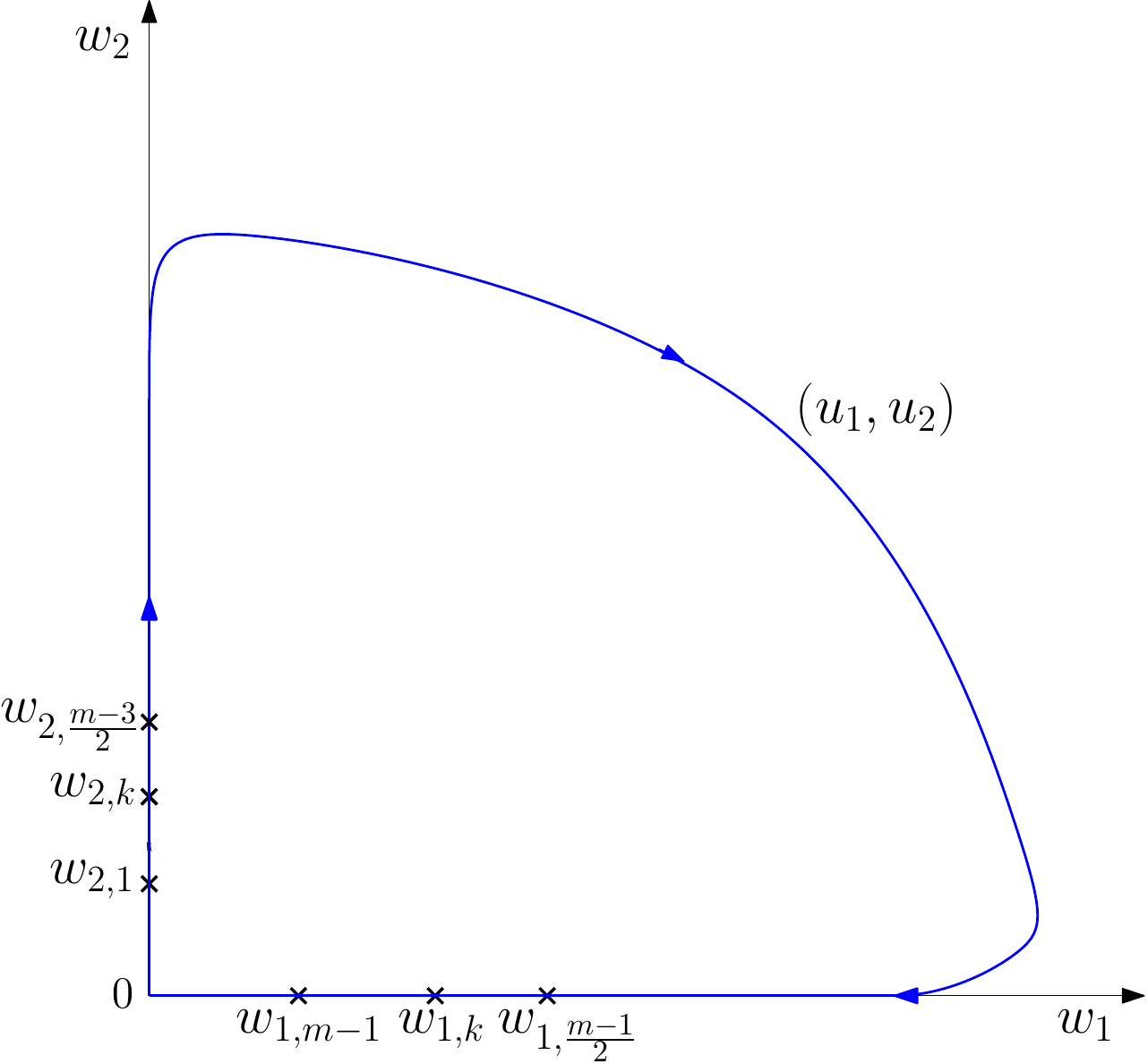}}
    \subfigure[Control path $(u_1,u_2)$ in the plane $(w_1,w_2)$ for $m$ even.]{\label{loopeven} \includegraphics[scale=0.58]{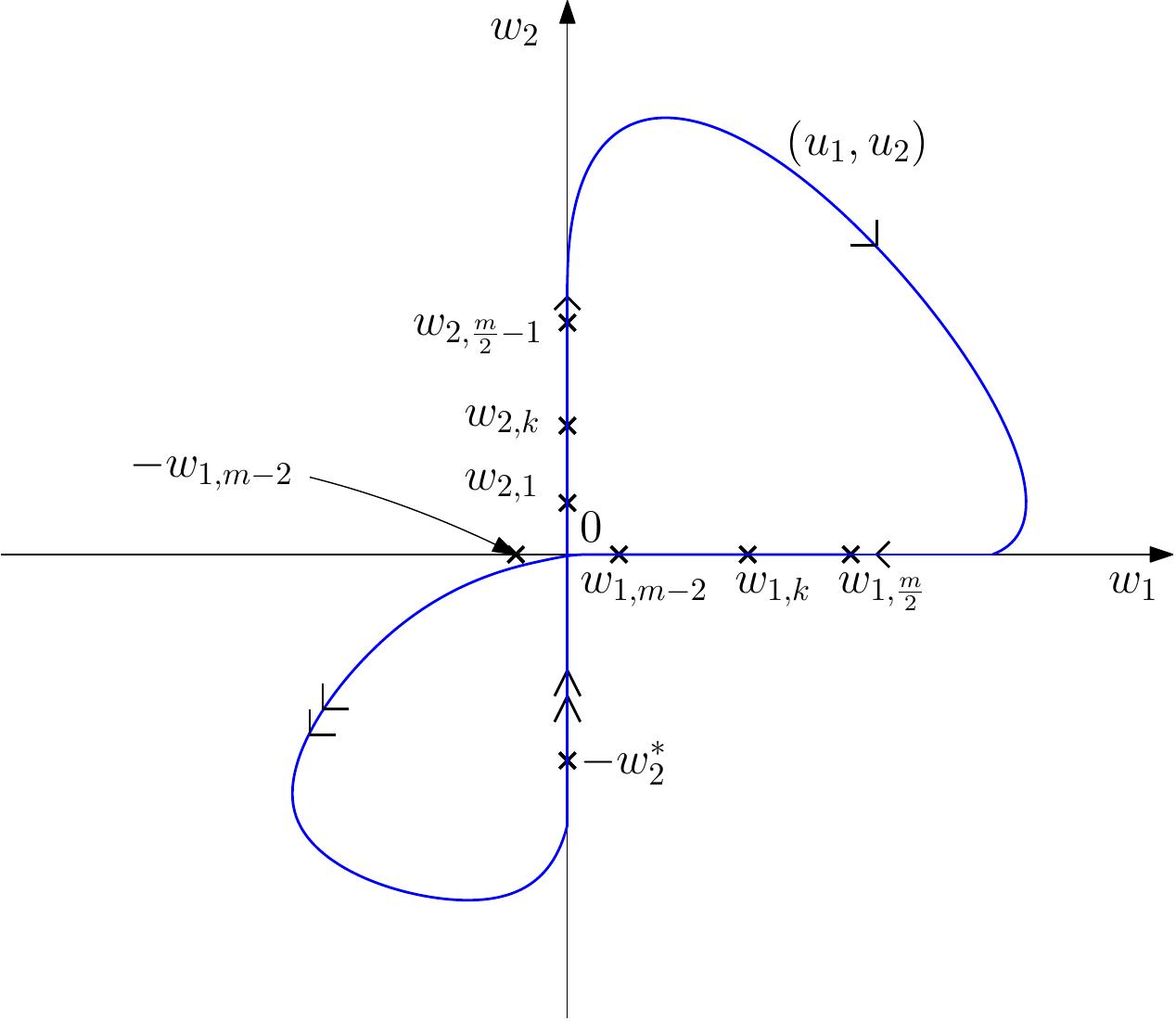}}
    \caption{}
\end{figure}

\begin{remark}[$m$ even]
In the case where $m$ is even we can state a result similar to Proposition~\ref{p:ssodd} by considering a control loop $(u_1,u_2)$ as in Figure~\ref{loopeven}. 
%During the control strategy which is proposed in Proposition~\ref{p:sseven}, according to Lemma~\ref{eig:int:ev}, 
The corresponding solution $\psi_{\epsilon}$ %with initial condition $e_1$ 
makes an approximate transition (up to phases) from $e_1$ to $e_{m-1}$ during the interval of time 
corresponding to the half-loop in the first quadrant and 
%$[0,\frac{1}{2\epsilon^{\alpha+1}}]$. 
then it makes an approximate transition from $e_{m-1}$ to $e_m$ 
when following 
%during the
% interval of time $[\frac{1}{2\epsilon^{\alpha+1}},\frac{1}{\epsilon^{\alpha+1}}]$.
%during the interval of time 
%corresponding to 
the half-loop in the third quadrant.
\end{remark}

\subsubsection{Simulations}\label{sec:sim2}
Define $H_0$ and $H_1$ as in Section~\ref{sec:sim1}.
%=\text{diag}(E_j)_{j=1}^7$ and $H_1$ be the $7$-dimensional symmetric matrix such that for every $j,k\in \{1,\dots 7\}$, $(H_1)_{j,k}=\begin{cases}1, \; \text{if}\; k=j+1\\
%0,\; \text{otherwise} \end{cases}$.
%
Let $(u_1,u_2)$ be chosen as in Proposition~\ref{p:ssodd}.
%Consider, for every $\tau\in [0,1]$, $u_1(\tau)=$ and $u_2(\tau)=$.
Let $\psi_\epsilon:[0,\frac{1}{\epsilon^{\alpha+1}}]\to\C^n$ be the solution of \eqref{Eq:gen} 
with initial condition $e_1$ associated with the control $u_{\epsilon}$. % and $(H_0,H_1)$ defined as previously.
%Define for $j\in \{1,\dots,7\}$, the \emph{population in level $j$} as $p_j(\tau)=|\langle \psi_{\epsilon}(\frac{\tau}{\epsilon^{\alpha+1}}), e_j \rangle|^2$ for $\tau\in [0,1]$.
%
We have plotted on Figure~\ref{st1} the 
population levels $p_j(\tau)=|\langle \psi_{\epsilon}(\frac{\tau}{\epsilon^{\alpha+1}}), e_j \rangle|^2$, $j=1,\dots,7$, 
%, of the
%solution of Equation~\eqref{Eq:gen} associated with the control $u_{\epsilon}$ and initial condition $e_1$ 
in the case $\alpha=1.2$, with $E_1=0, \ E_2=1, \ E_3 = 2.5, \ E_4 = 3, \ E_5 = 2.2, \ E_6 = 5, \ E_7 = 7$.
%, and for $\tau\in [0,1]$.

\begin{figure}[h!]
\center
\includegraphics[scale=0.7]{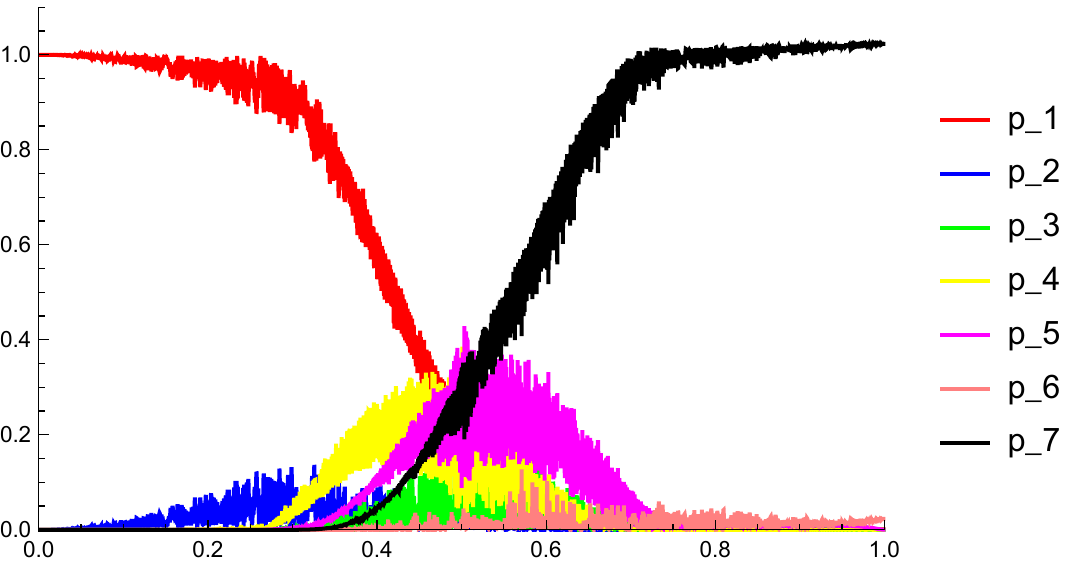}
\caption{Evolution of the different populations $(p_j(\tau))_{j=1}^7$ as functions of the renormalized time $\tau=\epsilon t \in [0,1]$, for $\alpha=1.2$ and $\epsilon=10^{-2}$.}
\label{st1}
\end{figure} 
%\begin{figure}[h!]
%\center
%\includegraphics[scale=1]{stfail.pdf}
%\caption{Evolution of the different populations $(p_j(\tau))_{j\in \{1,\dots, 7\}}$ as functions of the renormalized time $\tau=\epsilon t \in [0,1]$, $\alpha=0.8$ and $\epsilon=10^{-2}$.}
%\label{stfail}
%\end{figure} 
\begin{remark}%\red{[TO SHORTEN]}
In analogy with Remark~\ref{rem:robust-superchirp}, we can extend Proposition~\ref{p:ssodd} %and Remark~\ref{p:sseven}
 to the case where $H_1$ is replaced by $H_{1,\delta}=\delta H_1$, with $\delta\in [\delta_0,\delta_1]$, $0<\delta_0<\delta_1$.
\end{remark}

\begin{remark}
%\red{[TOO LONG TO DO ANGLES???]}
%\red{%ADVANTAGE OF SUPER STIRAP: 
%SPLITTINGS BY ANGLE METHOD AS IN [REF], 
%DO NOT DETAIL THE METHOD, MAY BE ONE 
%ADD A SIMULATION \\
%OF SUPERPOSITION BETWEEN STATE $1$ AND $n$
%ERCIM: 
We have focused in this and in the previous section on the control between eigenstates of the 
drift Hamiltonian $H_0$. As proposed in \cite{Bos} (see also \cite{yatsenko1})
broken adiabatic paths (with discontinuous first order derivatives at conical intersections) can be used to induce superpositions between eigenstates.
One could reason similarly for semi-conical intersections, using discontinuous second order derivatives and exploiting \cite[Proposition 17]{ABS19}.
 The control strategy presented here can therefore  be 
adapted to approximate, using oscillating controls, an adiabatic trajectory leading from an eigenstate of $H_0$ to a superposition of eigenstates with prescribed population levels. 
%the authors have developed a technique to create superpositions between eigenstates using broken adiabatic paths for a Hamiltonian driven by two real controls.
%The principle of the approach is that if
%an adiabatic path reaches a point in the
%singular set and makes there a corner,
%then the state splits partially on the
%lower energy level and partially on the
%upper one. Modulating the entry and
%exit direction, that is, the angle between
%the two, one can arbitrarily select the
%occupation of the two energy levels.
%Such an idea makes it possible, starting
%from an eigenstate, to reach any superposition
%of eigenstates whose corresponding
%eigenvalues intersect.
%In the following,
%}
\end{remark}

This work was supported 
by the ANR projects SRGI ANR-15-CE40-0018 and
%by the ANR project 
Quaco ANR-17-CE40-0007-01.

\bibliographystyle{abbrv}        
\bibliography{biblirwab}

\begin{thebibliography}{10}

\bibitem{adami-boscain}
R.~{Adami} and U.~{Boscain}.
\newblock Controllability of the {S}chr\"{o}dinger equation via intersection of
  eigenvalues.
\newblock In {\em Proceedings of the 44th IEEE Conference on Decision and
  Control}, pages 1080--1085, 2005.

\bibitem{Agra}
A.~A. Agrachev and Y.~L. Sachkov.
\newblock {\em Control theory from the geometric viewpoint}, volume~87 of {\em
  Encyclopaedia of Mathematical Sciences}.
\newblock Springer-Verlag, Berlin, 2004.
\newblock Control Theory and Optimization, II.

\bibitem{AKML1998}
D.~Alekseevsky, A.~Kriegl, P.~W. Michor, and M.~Losik.
\newblock Choosing roots of polynomials smoothly.
\newblock {\em Israel J. Math.}, 105:203--233, 1998.

\bibitem{Ensemble}
N.~Augier, U.~Boscain, and M.~Sigalotti.
\newblock Adiabatic ensemble control of a continuum of quantum systems.
\newblock {\em SIAM J. Control Optim.}, 56(6):4045--4068, 2018.

\bibitem{CDC-nicolas}
N.~Augier, U.~Boscain, and M.~Sigalotti.
\newblock On the compatibility between the adiabatic and the rotating wave
  approximations in quantum control.
\newblock In {\em Proceedings of the 58th IEEE Conference on Decision and
  Control}, 2019.

\bibitem{ABS19}
N.~Augier, U.~Boscain, and M.~Sigalotti.
\newblock Semi-conical eigenvalue intersections and the ensemble
  controllability problem for quantum systems.
\newblock {\em Math. Control Relat. Fields}, 2020.

\bibitem{boscain_approximate_2015}
U.~Boscain, J.-P. Gauthier, F.~Rossi, and M.~Sigalotti.
\newblock Approximate {Controllability}, {Exact} {Controllability}, and
  {Conical} {Eigenvalue} {Intersections} for {Quantum} {Mechanical} {Systems}.
\newblock {\em Communications in Mathematical Physics}, 333(3):1225--1239,
  2015.

\bibitem{Bos}
U.~V. Boscain, F.~Chittaro, P.~Mason, and M.~Sigalotti.
\newblock Adiabatic control of the {S}chr\"{o}dinger equation via conical
  intersections of the eigenvalues.
\newblock {\em IEEE Trans. Automat. Control}, 57(8):1970--1983, 2012.

\bibitem{TENS}
T.~{Chambrion}.
\newblock A sufficient condition for partial ensemble controllability of
  bilinear schrödinger equations with bounded coupling terms.
\newblock In {\em 52nd IEEE Conference on Decision and Control}, pages
  3708--3713, 2013.

\bibitem{Guerin98}
S.~Gu{\'e}rin and H.~Jauslin.
\newblock Two-laser multiphoton adiabatic passage in the frame of the {F}loquet
  theory. {A}pplications to (1+1) and (2+1) {STIRAP}.
\newblock {\em The European Physical Journal D - Atomic, Molecular, Optical and
  Plasma Physics}, 2(2):99--113, 1998.

\bibitem{Guerin99}
S.~Gu\'{e}rin, R.~G. Unanyan, L.~P. Yatsenko, and H.~R. Jauslin.
\newblock Floquet perturbative analysis for {STIRAP} beyond the rotating wave
  approximation.
\newblock {\em Opt. Express}, 4(2):84--90, 1999.

\bibitem{Irish}
E.~K. Irish.
\newblock Generalized rotating-wave approximation for arbitrarily large
  coupling.
\newblock {\em Phys. Rev. Lett.}, 99:173601, 2007.

\bibitem{Kurzweil88}
J.~Kurzweil and J.~Jarn\'{i}k.
\newblock Iterated {L}ie brackets in limit processes in ordinary differential
  equations.
\newblock {\em Results Math.}, 14(1-2):125--137, 1988.

\bibitem{RouchonSarlette}
Z.~Leghtas, A.~Sarlette, and P.~Rouchon.
\newblock Adiabatic passage and ensemble control of quantum systems.
\newblock {\em Journal of Physics B: Atomic, Molecular and Optical Physics},
  44(15):154017, 2011.

\bibitem{Liu97}
W.~Liu.
\newblock Averaging theorems for highly oscillatory differential equations and
  iterated {L}ie brackets.
\newblock {\em SIAM J. Control Optim.}, 35(6):1989--2020, 1997.

\bibitem{preprint-Remi}
R.~Robin, N.~Augier, U.~Boscain, and M.~Sigalotti.
\newblock On the compatibility of the adiabatic and rotating wave
  approximations for robust population transfer in qubits.
\newblock Preprint HAL-02504532.

\bibitem{SHAP}
E.~A. Shapiro, V.~Milner, and M.~Shapiro.
\newblock Complete transfer of populations from a single state to a preselected
  superposition of states using piecewise adiabatic passage: Theory.
\newblock {\em Phys. Rev. A}, 79:023422, 2009.

\bibitem{shore-book}
B.~W. {Shore}.
\newblock {\em {The Theory of Coherent Atomic Excitation, Volume 1, Simple
  Atoms and Fields}}.
\newblock 1990.

\bibitem{SHORE91}
B.~W. Shore, K.~Bergmann, J.~Oreg, and S.~Rosenwaks.
\newblock Multilevel adiabatic population transfer.
\newblock {\em Phys. Rev. A}, 44:7442--7447, 1991.

\bibitem{Teu}
S.~Teufel.
\newblock {\em Adiabatic perturbation theory in quantum dynamics}, volume 1821
  of {\em Lecture Notes in Mathematics}.
\newblock Springer-Verlag, Berlin, 2003.

\bibitem{Vitanov}
N.~V. Vitanov, T.~Halfmann, B.~W. Shore, and K.~Bergmann.
\newblock Laser-induced population transfer by adiabatic passage techniques.
\newblock {\em Annual Review of Physical Chemistry}, 52(1):763--809, 2001.

\bibitem{yatsenko1}
L.~Yatsenko, S.~Gu{\'e}rin, and H.~Jauslin.
\newblock Pulse-driven near-resonant quantum adiabatic dynamics: Lifting of
  quasidegeneracy.
\newblock {\em Physical Review A}, 70(4):043402, 2004.

\end{thebibliography}

\end{document}